\documentclass[english,11pt]{amsart}
\usepackage{amsmath}
\usepackage{latexsym}
\usepackage{amssymb}

\usepackage[colorlinks=true,linkcolor=black,citecolor=black]{hyperref}

\newcommand{\om}{\omega}
\newcommand{\lla}{\,\langle\!\langle\,}
\newcommand{\rra}{\,\rangle\!\rangle\,}
\newcommand{\la}{\,\langle\,}
\newcommand{\ra}{\,\rangle\,}

\newcommand{\R}{\mathbb{R}}

\newcommand{\C}{\mathbb{C}}
\newcommand{\Q}{\mathbb{Q}}
\newcommand{\A}{\mathcal{A}}
\newcommand{\Z}{\mathbb{Z}}

\newcommand{\wT}{\widetilde{T}}
\renewcommand{\wr}{\widehat{\rho}}
\newcommand{\wV}{\widehat{V}}
\newcommand{\wW}{\widehat{W}}
\newcommand{\wP}{\widehat{P}}
\newcommand{\wQ}{\widehat{Q}}
\newcommand{\wE}{\widehat{E}}

\renewcommand{\H}{\mathcal{H}}
\renewcommand{\le}{\leqslant}
\renewcommand{\ge}{\geqslant}

\newcommand{\res}{\mathrm{Res }}

\newcommand{\reg}{\mathrm{reg}}

\newcommand{\F}{\mathcal{F}}

\newcommand{\gl}{\mathrm{GL}}
\newcommand{\ind}{\mathrm{Ind}}
\newcommand{\re}{\mathrm{Re}}\newcommand{\im}{\mathrm{Im}}

\newcommand{\tr}{\mathrm{Tr}}
\newcommand{\SL}{\mathrm{SL}}

\newcommand{\cI}{\mathcal{I}}

\newcommand{\D}{\mathcal{D}}

\renewcommand{\SS}{\mathcal{S}}

\newcommand{\EE}{\mathcal{E}}

\newcommand{\nt}{\,\widetilde{n}}
\newcommand{\tf}{\widetilde{f}}
\newcommand{\tF}{\widetilde{F}}
\newcommand{\tV}{\widetilde{V}}
\newcommand{\ov}[1]{\overline{#1}}
\newcommand{\og}{{\overline{\Gamma}}}

\newcommand{\cc}{\mathcal{C}}
\newcommand{\comment}[1]{}

\theoremstyle{plain}
\newtheorem{theorem}{Theorem}[section]
\newtheorem{corollary}[theorem]{Corollary}
\newtheorem{prop}[theorem]{Proposition}
\newtheorem{lemma}[theorem]{Lemma}

\numberwithin{equation}{section}
\theoremstyle{definition}
\newtheorem{remark}[theorem]{Remark}
\newtheorem{definition}[theorem]{Definition}

\title{Modular forms and period polynomials}
\author{Vicen\c{t}iu Pa\c{s}ol, Alexandru A. Popa}

\address{Institute of Mathematics ``Simion Stoilow" of the Romanian Academy,
P.O. Box 1-764, RO-014700 Bucharest, Romania}
\address{E-mail: vicentiu.pasol@imar.ro}
\address{E-mail: alexandru.popa@imar.ro}

\begin{document}

\begin{abstract}We study the space of period polynomials associated with modular forms of
integral weight for finite index subgroups of the modular group. For the modular group, this space is endowed
with a pairing, corresponding to the Petersson inner product on modular forms via a formula of
Haberland, and with an action of Hecke operators, defined algebraically by Zagier. We generalize
Haberland's formula to (not necessarily cuspidal) modular forms for finite index subgroups, and we show that it
conceals two stronger formulas. We extend the action of Hecke operators to period polynomials of modular forms, we
show that the pairing on period polynomials appearing in Haberland's formula is nondegenerate, and we determine
the adjoints of Hecke operators with respect to it. We give a few applications for $\Gamma_1(N)$: an extension of
the Eichler-Shimura isomorphism to the entire space of modular forms; the determination of the 
relations satisfied by the even and odd parts of period polynomials associated with cusp forms,
which are independent of the period relations; and an explicit formula for Fourier coefficients of Hecke eigenforms
in terms of their period polynomials, generalizing the Coefficients Theorem of Manin. 
\end{abstract}
\maketitle

\section{Introduction}

Let $\Gamma$ be a finite index subgroup of $\Gamma_1=\SL_2(\Z)$, and
let $S_k(\Gamma)$ be the space of cusp forms of integer weight $k\ge 2$ for $\Gamma$. Let $V_w$ be
the $\Gamma$-module of complex polynomials of degree at most $w=k-2$.
Each form $f\in S_k(\Gamma)$ determines a collection of polynomials $\rho_f: \Gamma\backslash
\Gamma_1\rightarrow V_w$ given by
\[
\rho_f(A)=\int_0^{i\infty} f|_k A(t) (t-X)^w dt.
\]
The object $\rho_f$ belongs to the induced $\Gamma_1$-module $\ind_\Gamma^{\Gamma_1}
V_w$, and
we call it the (multiple) period polynomial associated to $f$.  The goal of this paper is to
investigate the structure of the space of period polynomials, reflecting the Petersson inner
product and the Hecke operators on modular forms. Working inside the subspace  of
period polynomials $W_w^\Gamma\subset \ind_\Gamma^{\Gamma_1} V_w$, we
show that the Petersson product, and the action of Hecke operators, can be
stated in a
simple way in terms of period polynomials. On an abstract level, this is explained by the fact
that the parabolic cohomology  class associated to $f$ 
is completely determined by $\rho_f$, as reviewed in Section \ref{sec2} where we restate the
Eichler-Shimura isomorphism in terms of period polynomials. Our results  can be
interpreted as translating the cup product and the action of Hecke operators on cohomology,
into a pairing and a Hecke action on the space of period polynomials. 

An essential ingredient in our approach is a generalization of a formula of Haberland
expressing the Petersson product of two cusp forms for the modular group in terms of a pairing on their period
polynomials \cite{H,KZ}. In Section \ref{sechab} we show that Haberland's formula can be extended to finite
index subgroups of $\Gamma_1$. More importantly, using an involution on period polynomials that corresponds to
complex conjugation, we show that Haberland's formula splits in two simpler formulas, pairing
the opposite sign parts (respectively the same sign parts) of the period polynomials of
the two forms when $k$ is even (respectively when $k$ is odd). For the full
modular group, the stronger formulas were proved by different means in \cite{Po}. 
When $k=2$, a generalization of Haberland's formula was given by Merel \cite{M09}, and a proof for finite
index subgroups and arbitrary weight was very recently given by Cohen \cite{Co}.
Our proof is simplified by the use of Stokes' theorem on a fundamental domain for $\Gamma(2)$,
which clarifies the appearence of the period polynomial pairing in the formula.

The action of Hecke operators on period polynomials was defined algebraically by Zagier for
the full modular group \cite{Z90,Za93,CZ}. It was extended by Diamantis to operators of index coprime with the
level for the congruence subgroups $\Gamma_0(N)$ \cite{Di01}. We show in Section
\ref{sec4} that the same elements as in the full level case, which go back to work of Manin
\cite{M73}, have actions on period polynomials that correspond to actions of a large class of double coset
operators on modular forms, including Hecke and Atkin-Lehner operators for $\Gamma_1(N)$. We also determine
the adjoint of the Hecke action with respect to the pairing on period polynomials appearing in Haberland's formula.
The proof of Hecke equivariance given here relies on the generalization of Haberland's formula, and a completely
algebraic proof is given in \cite{PP12a}. 

As an application of the action of Hecke operators on period polynomials, we obtain a generalization of the
Coefficients Theorem of Manin, giving the Fourier coefficients of a Hecke eigenform for $\Gamma_1(N)$ in terms of
its even period polynomial. We also give a simple proof of the rationality of period polynomials of Hecke
eigenforms for $\Gamma_1(N)$ in \S\ref{sec5.1}. We discuss period polynomials of cusp forms with
nontrivial Nebentypus in \S\ref{sec5.2}. Our results can be used to efficiently compute period polynomials of Hecke
eigenforms numerically, as well as Hecke eigenvalues and Petersson norms, and we give an example in
\S\ref{sec5.3}. 

We give two applications of the stronger form of Haberland's formula for cusp forms: in Section
\ref{sec6} we prove a decomposition of cusp forms in terms of Poincar\'e series generators;
while in Section \ref{sec7} we obtain the extra relations satisfied
by the even and odd period polynomials of cusp forms, obtained by Kohnen and Zagier in the
full level case \cite{KZ}. For $\Gamma=\Gamma_0(N)$, we characterize those $N$ for which the
extra relations involve only the even parts of period polynomials just like in the full
level case, that is those $N$ for which the map $\rho^-:S_k(\Gamma)\rightarrow (W_w^\Gamma)^-$ is an isomorphism
(Prop. \ref{prop7.2}). The extra relations are explicit once the period polynomials
of the generators with rational periods are computed, as partially done in \cite{An}, \cite{FY}.
For small $N$ that is enough to give completely explicit relations, and we illustrate this
for $\Gamma_0(2)$.

In the last section, we define the space $\wW_w^\Gamma$ of period polynomials of all modular forms, following the
construction for the modular group in \cite{Z91}. We generalize Haberland's formula and its
refinement to this larger space, and we show that the pairing appearing in Haberland's formula is nondegenerate on
$\wW_w^\Gamma$. In contrast, on $W_w^\Gamma$ the radical of this pairing consists of the ``coboundary
polynomials'', of dimension equal to the dimension of the Eisenstein subspace of $M_k(\Gamma)$, as shown in Section
\ref{sec_cw}. 

For $\Gamma=\Gamma_1(N)$ and $k>2$, we show
that the plus and minus period polynomial maps extend to isomorphisms  $\rho^{\pm}:M_k(\Gamma)
\rightarrow(\wW_w^\Gamma)^\pm$. This can be seen as an extension of the
Eichler-Shimura isomorphism to the entire space of modular forms. Surprisingly, when
$k=2$ the two maps are not always isomorphisms: for $\Gamma=\Gamma_0(N)$ with $N$ square free
with at least two prime factors, precisely one of the two maps is an isomorphism (Proposition
\ref{p7.4} and Remark \ref{r7.5}). In this context we point out that Haberland's formula has been generalized
to weakly holomorphic modular forms of full level in \cite{BGKO}, and it would be interesting to investigate if
the results proved here for modular forms hold in that setting as well. 

In \S\ref{sec7.1} we extend the action of Hecke operators to the space of period polynomials of all modular
forms, and we show that the adjoints of Hecke operators on the larger space $\wW_w^\Gamma$ are the same as on
$W_w^\Gamma$. As an application, for $\Gamma=\Gamma_1(N)$ we show that for $(n,N)=1$ 
\[\tr (W_w^\Gamma |_\Delta \wT_n ) = \tr ( M_k(\Gamma)| T_n)+ \tr ( S_k(\Gamma)| T_n)
\]
where the $|_\Delta \wT_n$ is the action on period polynomials corresponding to the action of the Hecke operator
$T_n$ on modular forms. A similar statement holds for modular forms with Nebentypus, and for traces
of Atkin-Lehner operators on $\Gamma_0(N)$. This fact is used in upcoming joint work of
the second author with Don
Zagier to give a simple proof of the Eichler-Selberg trace formula for modular forms on $\Gamma_0(N)$ with
Nebentypus.

A second method of obtaining explicit extra relations among periods of cusp forms 
is sketched in Section \ref{sec7.2}. It generalizes the $\Gamma_1$ approach in \cite{KZ}, by
using period polynomials of Eisenstein series and Haberland's formula for arbitrary modular
forms.

We note that period polynomials are dual to modular symbols, in the sense that the coefficients
of period polynomials are values of the integration pairing between modular forms and Manin
symbols (the duality between cohomology and homology). The results of this paper are
therefore parallel to the modular symbol formalism developed by Merel \cite{Me}, and they also lead
to efficient algorithms for modular form computations, as shown in \S\ref{sec5.3}. An additional structure that
we introduce here is the extended pairing on $\wW_w^\Gamma$, whose nondegeneracy and Hecke equivariance
properties are important even if one is interested only in period polynomials of cusp forms. For example, the
properties of the extended pairing were used in the computation of the Hecke and Atkin-Lehner traces on
$W_w^\Gamma$ mentioned above, and in the proof of rationality of $\rho_f^\pm$ for newforms $f\in S_k(\Gamma_1(N))$
(Prop. \ref{p5.7}). 

The paper is self-contained, except for using the fact that the dimension of the parabolic cohomology group
$H^1_P(\Gamma, V_w)$ equals twice the dimension of $S_k(\Gamma)$, a consequence of the Eichler-Shimura
isomorphism. 

\comment{
The paper is organized as follows. In Section \ref{sec2} we restate the
Eichler-Shimura isomorphism in terms of period polynomials, and introduce most of the
notations used throughout. In Sections \ref{sechab}-\ref{sec7} we consider
only period polynomials of cusp forms, while in Section \ref{sec8} we generalize most of the
results of the previous sections to period polynomials of arbitrary modular forms. }

\section{Period polynomials and the Eichler-Shimura isomorphism}\label{sec2}

The theory of period polynomials for $\Gamma_0(N)$ has been treated in
\cite{sk90,An,Di01}. We review it here in a general setting, and interpret the Eichler-Shimura
isomorphism in terms of period polynomials. We use the properties of the pairing on period
polynomials introduced in Section \ref{sechab} to prove injectivity of the Eichler-Shimura
map. In this section we fix notations in use throughout the paper.

Let $\Gamma$ be a finite index subgroup of $\Gamma_1=\SL_2(\Z)$, and denote by
$\og=\Gamma/(\Gamma\cap \{\pm 1\})$ the projectivisation of $\Gamma$.  Throughout
the paper, the weight $k\ge 2$ is an integer, and we set $w=k-2$. Let $V_w$ be the module
of complex polynomials of degree at most $w$, with (right) $\Gamma_1$-action by the $|_{-w}$
operator: $P|_{-w} g (z)= P(gz) j(g,z)^w$ where $j(g,z)=cz+d$ for $g=\left(\begin{smallmatrix}
* & * \\ c & d \end{smallmatrix}\right)\in \gl_2(\R)$. Since this is the only
action on polynomials, we will omit the subscript. 

Viewing $V_w$ as a $\Gamma$-module, let $\tV_w^\Gamma$ be the induced
$\Gamma_1$-module $\ind_{\Gamma}^{\Gamma_1}(V_w)$.
Since $V_w$ is also a $\Gamma_1$-module, 
we can identify $\tV_w^\Gamma$ with the space of maps  $ P :\Gamma\backslash \Gamma_1
\rightarrow V_w$ with $\Gamma_1$ action:
\[ 
P|g (A) = P(Ag^{-1})|_{-w}g.    
\] 
By the Shapiro isomorphism, we have $H_P^1(\Gamma, V_w)\simeq H_P^1(\Gamma_1,
\tV_w^\Gamma)$ (parabolic cohomology groups). For background on Shapiro's lemma and induced
modules, see \cite[p.59]{NSW}. 

Letting $J=\left(\begin{smallmatrix} -1 & 0 \\ 0 & -1
\end{smallmatrix}\right)$, for any cocycle $\sigma: \Gamma_1\rightarrow
V_w^\Gamma$ we have $\sigma(g)|(1-J)=\sigma(J)|(1-g)$  for all $g\in \Gamma_1$(which follows from
$\sigma(Jg)=\sigma(gJ)$). It follows that the cocycle
$\tilde{\sigma}=\frac{\sigma+\sigma|J}{2}$ is in the same cohomology class as $\sigma$,
where $\sigma|J(g):=\sigma(g)|J$. Since the cocycle $\tilde{\sigma}$ takes values in the
subspace
$$
V_w^\Gamma:=\{P\in \tV_w^\Gamma: P|J=P, \text{ that is } P(A)=(-1)^w P(-A) \}
$$
we have $H_P^1(\Gamma_1, \tV_w^\Gamma)\simeq H_P^1(\Gamma_1, V_w^\Gamma) $ and from now on we
will only work inside the space $V_w^\Gamma$. Note that when $k$ is even, $V_w$ is both a
$\og$ and a $\og_1$ module, and the space $V_w^\Gamma$ can be identified with
$\ind_{\og}^{\og_1}(V_w)$.

Let now $f\in S_k(\Gamma)$, and define a cocycle $\sigma_f:\Gamma_1\rightarrow V_w^\Gamma$
by:
\[
\sigma_f(g) (A) =\int_{g^{-1}i\infty}^{i\infty} f|A(t) (t-X)^w dt,
\]
where the stroke operator acting on modular forms of weight $k$ is $f|g=f|_k g$ for
$g\in GL_2(\R)^+$.  The action of
the coset $A$ is defined
by acting with any coset representative; this is independent of the representative chosen since
$f|_k\gamma=f$ for $\gamma\in \Gamma$. We will show at the end of this section that
$\sigma_f$ satisfies the cocycle relation
\[
\sigma_f(g_1 g_2)=\sigma_f(g_2)+\sigma_f(g_1)|g_2.
\]

Let $S=\left(\begin{smallmatrix} 0 & -1 \\ 1 & 0 \end{smallmatrix}\right), \
T=\left(\begin{smallmatrix} 1 & 1 \\ 0 & 1 \end{smallmatrix}\right)
$, and let $U=TS$, so that $U^3=J$. Clearly $\sigma_f(\pm T^n)$ vanishes for $n\in \Z$,
and it is easy to see (by a change of variables) that it is a coboundary for other
parabolic elements of $\Gamma_1$, hence $\sigma_f$ defines an element $[\sigma_f]
\in H_P^1(\Gamma_1,
V_w^\Gamma)$. Since $T$, $S$ and $J$ generate $\Gamma_1$, it follows that the cohomology class $[\sigma_f]$ is
completely determined by the value $\rho_f=\sigma_f(S)\in V_w^\Gamma$, which is the
multiple period polynomial attached to $f$ in the introduction. Using the fact that
$\sigma_f(S^2)=\sigma_f(U^3)=\sigma_f(US)=0$ and the cocycle
relation, it follows that $\rho_f$ satisfies the period polynomial relations:
\[
\rho_f|(1+S)=0, \ \ \ \rho_f|(1+U+U^2)=0.
\]
We also have $\rho_f(-A)=(-1)^w\rho_f(A)$, so $\rho_f|J= \rho_f$. Therefore the image of
the map $f\rightarrow \rho_f$ is contained in the subspace\footnote{The condition $P|J=P$ is
part of the definition of $V_w^\Gamma$, but we include it for clarity.}
$$
W_w^\Gamma=\{P\in V_w^\Gamma\ :\ P|(1+S)=0, \ \ P|(1+U+U^2)=0, \ \ P|J= P \} 
$$
whose elements we call \emph{period polynomials} (each element is in fact a collection of
$[\Gamma_1:\Gamma]$ polynomials belonging to $V_w$). 

In fact, setting $C_w^\Gamma=\{P|(1-S) : P \in
V_w^\Gamma,\ \ P|T=P\}\subset W_w^\Gamma,$ we have an isomorphism 
\begin{equation}\label{2.2}
H_P^1 (\Gamma_1, V_w^\Gamma)\simeq W_w^\Gamma/C_w^\Gamma,
\end{equation}
obtained by choosing representative cocycles $\sigma$ such that $\sigma( T)=\sigma(J)=0$, and
sending $[\sigma]$ to $\sigma(S)\in W_w^\Gamma$. The space $C_w^\Gamma$ is the image
of coboundaries, and we show in Lemma \ref{L7.1} that its dimension equals the dimension of
the Eisenstein subspace $\EE_k(\Gamma)$ of $M_k(\Gamma)$. 

Assume now that $\Gamma$ is normalized by 
$\epsilon=\big(\begin{smallmatrix} -1 & 0 \\
	 0 & 1 \end{smallmatrix}\big)$. The matrix $\epsilon$ acts on $P\in V_w^\Gamma$ by
\begin{equation}\label{eps}
P|\epsilon (A) = P(A')|_{-w} \epsilon,
\end{equation}
where $A'=\epsilon A \epsilon$. This action is compatible with the action of $\Gamma_1$: $P| g
|
\epsilon = P|\epsilon|\epsilon g \epsilon$ for all $g\in
\Gamma_1$. If $f^*\in S_k(\Gamma)$ denotes the form $f^*(z)=\overline{f(-\overline{z})}$, then
\begin{equation}\label{2.1}
 \overline{\rho_{f^*}}=(-1)^{w+1}\rho_f|\epsilon
\end{equation}
where  $\overline{P}(A)$ is obtained by taking the complex conjugates of the coefficients
of $P(A)$.

Under the action of $\epsilon$, the space $V_w^\Gamma$ breaks into $\pm
1$-eigenspaces, denoted by  $(V_w^\Gamma)^\pm$. For $P\in V_w^\Gamma$ we denote its
+1 and -1-components by $P^+$ and $P^-$ respectively:  
\begin{equation}\label{e_ev}
P^{\pm}=\frac{1}{2}( P \pm P|\epsilon ) \in (V_w^\Gamma)^\pm.
\end{equation}
We call $P^+$ the \emph{even part} and $P^-$ the \emph{odd part} of $P$,
which is justified by the fact that $P(I)^+$ is an even polynomial, and $P(I)^-$ is an odd 
polynomial ($I$ is
the identity coset). For $P\in W_w^\Gamma$, it is easily checked that $P|\epsilon \in
W_w^\Gamma$ as well. Therefore $P^+, P^-\in W_w^\Gamma$, and the space $W_w^\Gamma$ also
decomposes into eigenspaces  $(W_w^\Gamma)^\pm$. 

Making use of the pairing on period polynomials introduced in Section \ref{sechab}, we restate
the Eichler-Shimura isomorphism in terms of period polynomials as follows. 
\begin{theorem}[Eichler-Shimura] \label{thm2.1}
The two maps $\rho^{\pm}:S_k(\Gamma) \rightarrow
(W_w^\Gamma)^{\pm}$, $f\mapsto \rho_f^\pm$, give rise  to
isomorphisms, denoted by the same symbols:
 \begin{equation}\label{7.1}
\rho^\pm:S_k(\Gamma) \longrightarrow (W_w^\Gamma)^{\pm}/(C_w^\Gamma)^{\pm}.
\end{equation}
\end{theorem}
\begin{proof}
By the stronger version of Haberland's formula (Theorem \ref{thm_main}), the two
maps $\rho^{\pm}:S_k(\Gamma) \rightarrow (W_w^\Gamma)^{\pm}$ are injective. Moreover, their
images intersect trivially with $C_w^\Gamma$ by Lemma \ref{l4.4}, so the two maps in \eqref{7.1} are
also injective. Using \eqref{2.2} and the
Eichler-Shimura isomorphism \cite[Ch. 8]{Sh} we have $\dim W_w^\Gamma= 2\dim S_k(\Gamma)+\dim
C_w^\Gamma$, and we conclude that $\rho^\pm$ in \eqref{7.1} are isomorphisms.
\end{proof}

We now show that $\sigma_f$ satisfies the cocycle relation, while also giving another
construction of the associated period polynomial. In analogy with the $\Gamma_1$ case, the
``Eichler
integral'' associated with $f\in S_k(\Gamma)$ is a function
$\tf:\Gamma\backslash\Gamma_1\rightarrow
\A$, where $\A$ is the space of holomorphic functions on the upper half plane, given by:
\begin{equation}\label{eq_eich}
\tf(A) (z)=\int_z^{i\infty} f|A(t) (t-z)^w dt, 
\end{equation}
with $\Gamma_1$-action as on period polynomials: 
$\tf|g (A)= \tf(Ag^{-1})|_{-w}g$ for $g\in \Gamma_1$. By a change of variables we
see that $ \tf|(1-g)=\sigma_f(g)$, which implies that $\sigma_f$ satisfies the cocycle
relation. Note that this provides another
construction for the period polynomial $\rho_f$ attached to $f$, which we record for further
use: 
\begin{equation}\label{e_int}
 \rho_f=\tf|(1-S).
\end{equation}\vspace{-5mm}
\begin{remark}
A similar construction will be used in Section \ref{sec8} to define period polynomials of 
arbitrary modular forms, by means of an Eichler integral $\tf$ of $f\in M_k(\Gamma)$,
which has the property that $\tf| (1-T)=0$ and $\tf|(1-S)$ is the (extended) period
polynomial
attached to $f$. As pointed out in \cite{DIT}, the construction of period polynomials of
cusp forms using their higher order integrals goes back to Poincar\'e.  
\end{remark}

\section{Generalization of Haberland's formula} \label{sechab}

In \cite{H}, Haberland proved a formula expressing the Petersson product of two
cusp forms for the full modular group in terms of a pairing on their period polynomials.
In this section we extend Haberland's formula  to a finite index subgroup
$\Gamma$ of $\Gamma_1$, and we prove a stronger version for subgroups normalized by
$\epsilon$. 

For $f,g\in S_k(\Gamma)$, define the Petersson scalar product:
\[(f,g)=\frac{1}{[\og_1:\og]}\int_{\Gamma\backslash\H} f(z)
\overline{g(z)} y^k
\frac{dx dy}{y^2}.
\]

On $V_w\times V_w$ we have a natural pairing 
$\la \textstyle\sum a_n x^n, \sum b_n x^{n}\ra =\sum (-1)^{w-n} \binom{w}{n}^{-1} a_n b_{w-n},$
satisfying $\la P,Q \ra=(-1)^w\la Q,P\ra$. We will mostly use the equivalent formulation 
\begin{equation}\label{3.3}
\la (aX+b)^w,(cX+d)^w\ra =(ad-bc)^w.
\end{equation}
An easy consequence of \eqref{3.3} is that $\la P|g,Q\ra=\la P, Q|g^\vee\ra$ for $g\in
\gl_2(\R)$, where $g^\vee=g^{-1} \det g$; in particular the pairing is $\SL_2(\R)$-invariant.

We define a similar pairing on $V_w^\Gamma \times V_w^\Gamma$: 
\begin{equation}\label{3.0}
\lla P, Q \rra=\frac{1}{[\og_1:\og]}\sum_{A\in \og\backslash \og_1} \la P(A),Q(A)\ra \ \
\text{ for  } P,Q \in
V_w^\Gamma.
 \end{equation}
\begin{remark} \label{r_sign} 
For odd $k$ there is a sign ambiguity in defining $P(A)$, $Q(A)$ for $P,Q \in V_w^\Gamma$ and
$A\in \og\backslash\og_1$, but the pairing is well-defined since $P(-A)=(-1)^w P(A)$,
$Q(-A)=(-1)^w Q(A)$. For the same reason, one can replace the range  by $A\in
\Gamma\backslash \Gamma_1$ and the normalizing factor by $\frac{1}{[\Gamma_1:\Gamma]}$ without
changing the pairing. A similar observation applies below, when $f|A$ always appears paired
with $\ov{g}|A$, for $f,g\in S_k(\Gamma)$.
\end{remark}

This pairing is $\Gamma_1$-invariant: $ \lla P|g, Q|g \rra=\lla P,Q\rra$, for all $P,Q \in
V_w^\Gamma $
and $g\in \Gamma_1$. It is normalized such that if $f,g \in S_k(\Gamma)$ and $\Gamma'\subset
\Gamma$ then $\lla\rho_f, \rho_g \rra_{\Gamma} = \lla\rho_f, \rho_g \rra_{\Gamma'}$.
Define also the modified pairing on $V_w^\Gamma \times V_w^\Gamma$:
\begin{equation}\label{pairing}
 \{P,Q\}=\lla P|T-T^{-1}, Q\rra,
\end{equation}
which satisfies $\{P,Q\}=(-1)^{w+1}\{Q,P\}$. 

Part a) of the following
theorem generalizes Haberland's formula \cite{H,KZ}. Part b) follows easily from the
proof of part a), although to our knowledge it has not appeared previously in the literature
(except for $\Gamma_1$ in \cite{Po}, but the proof there is more complicated). 
\begin{theorem}\label{thm_hab} (a) For $f,g\in S_k(\Gamma)$, we have
\[ 
6 C_{ k} \cdot(f,g)=\{ \rho_f , \overline{\rho_g}\},
\] 
where complex conjugation acts coefficientwise on polynomials and $C_{k}= -(2i)^{k-1}$.  

(b) For $f,g\in S_k(\Gamma)$, we have $\{ \rho_f , \rho_g\}=0$. 
\end{theorem}
\begin{proof} 
(a) The proof is based on Stokes' theorem, as in
\cite{KZ}, except that we apply it to a fundamental domain for $\Gamma(2)$, namely the
quadrilateral region $\D$ with vertices $i\infty, -1, 0, 1$ and with sides the geodesics
connecting the points in that order. The region $\D$ consists of six copies of the fundamental
domain for $\Gamma_1$, which explains the constant 6 appearing in the formula.  Therefore we
have:
\begin{equation}\label{3.1}
\begin{aligned}
 6 C_{k}[\og_1:\og]\cdot (f,g) &= 6 \int_{\Gamma\backslash \H} f(z)
\overline{g(z)}(z-\overline{z})^w
dz d\overline{z}\\
  &= \sum_{A\in\og \backslash \og_1} \int_\D f|A(z) \overline{g|A} (z)
(z-\overline{z})^w dz d\overline{z}\\ 
  &=\sum_{A\in\og \backslash \og_1} \int_{\partial\D}  F_A(z) \overline{g|A} (z)
d\overline{z}
\end{aligned} 
\end{equation}
where $F_A(z)=\int_{i\infty}^{z} f|_k A(t)(t-\overline{z})^w dt$, so that
$\frac{\partial F_A}{\partial z}=f|A(z) (z-\overline{z})^w$, and the last line follows from
Stokes' theorem. For $B\in \SL_2(\Z)$ a change of variables shows that
\begin{equation}\label{3.2}
j(B, \overline{z})^w F_A (Bz)=F_{AB}(z) - \int_{i\infty}^{B^{-1} i\infty} f|AB
(t)(t-\overline{z})^w dt.
\end{equation}
We denote by $\int_a^b$ the integral over the geodesic arc from the cusp
$a$ to the cusp $b$. A change of variables $z=T^2 \tau$ and \eqref{3.2}
yields: 
\[
\int_{1}^{i\infty}  F_A(z) \overline{g|A} (z)
d\overline{z} = \int_{-1}^{ i\infty} F_{A T^2}(\tau) \overline{g|AT^2} (\tau)
d\overline{\tau}, 
\]
and it follows that the sum of integrals over the vertical sides of $\D$ vanishes. 

A change of variables $z=S\tau$ and \eqref{3.2} yields:
\[
\int_{-1}^{0}  F_A(z) \overline{g|A} (z)
d\overline{z} = \int_{1}^{i\infty} F_{A S}(\tau) \overline{g|A S} (\tau)
d\overline{\tau} + \int_{1}^{i \infty} \int_0^{i\infty} f|AS (t) (t-\overline{\tau})^w
\overline{g|AS(\tau)} dt d\overline{\tau}.
\]
\[
\int_{0}^{1}  F_A(z) \overline{g|A} (z)
d\overline{z} = \int_{i\infty}^{-1} F_{A S}(\tau) \overline{g|A S} (\tau)
d\overline{\tau} - \int_{-1}^{i \infty} \int_0^{i\infty} f|AS(t) (t-\overline{\tau})^w
\overline{g|AS(\tau)} dt d\overline{\tau}.
\]
When adding the last two equations and summing over $A \in\og \backslash \og_1$,
the single integrals cancel as before and \eqref{3.1} becomes
\[
6 C_{k} [\og_1:\og]\cdot (f,g)=\sum_{A\in\og \backslash \og_1}
\int_{1}^{i \infty}
\int_0^{i\infty}-  \int_{-1}^{i \infty} \int_0^{i\infty}  
f|A (t) (t-\overline{\tau})^w \overline{g|A(\tau)} dt d\overline{\tau}.
\]
To write the double integrals in terms of the period polynomial pairing, we use \eqref{3.3}.
After a change of variables $\tau=Tz$ the first double
integral becomes
\[
\int_{0}^{i \infty} \int_0^{i\infty}f|A (t) \left\langle (t-X)^w, (\overline{Tz}-X)^w
\right\rangle 
 \overline{g|AT(z)}\ d t d\overline{z}=\la \rho_f(A), \overline{\rho_g}
(AT)|T^{-1}\ra .
\]
The second integral yields the same result, with $T$ replaced by $T^{-1}$, and
the conclusion follows from the fact that the pairing $\lla , \rra$ is $\Gamma_1$ invariant.

\noindent (b) Going backwards in the proof of part (a) we have:
\[
\{\rho_f, \rho_g\} = \frac{1}{[\og_1:\og]} \sum_{A\in\og \backslash \og_1} \int_{\partial\D} 
H_A(z) g|A (z)
d z
\]
where $H_A(z)=\int_{i\infty}^{z} f|_k A(t)(t-z)^w dt=-\tf(A)(z)$. Since the integrand is now
holomorphic and vanishes exponentially at the cusps, it follows that each integral above
vanishes. 
\end{proof}

Now let $\Gamma$ be a congruence subgroup normalized by $\epsilon$. The pairing  $\{\cdot,
\cdot \}$ satisfies
\begin{equation}\label{conj1}
\{P|\epsilon, Q|\epsilon\} = (-1)^{w+1}\{P,Q\}, 
\end{equation}
hence $\{P,Q\}=0$ if $k$ is even and $P, Q\in V_w^\Gamma$ have the same parity, or if $k$ is
odd and $P,Q$ have opposite parity. We have the following stronger
version of Haberland's theorem, generalizing the result for the full modular group from
\cite{Po}. 
\begin{theorem} \label{thm_main}
Let $\Gamma$ be a subgroup of finite index in $\Gamma_1$, normalized by $\epsilon$. For $f,g\in
S_k(\Gamma)$:
\[ 
3 C_{k}\cdot (f,g)=\{\rho_f^{\kappa_1} , \overline{\rho_g^{\kappa_2}}\}
\]
for any $\kappa_1, \kappa_2\in\{+, -\}$ with $\kappa_1\ne \kappa_2$ if $k$ even
and $\kappa_1=\kappa_2$ if $k$ odd.
\end{theorem}
\begin{proof} We assume $k$ even, the case $k$ odd being entirely
similar. In view of Theorem \ref{thm_hab} (a) and \eqref{conj1}, it is enough to show that 
$\{\rho_f^+ , \overline{\rho_g^-}\}=\{\rho_f^- , \overline{\rho_g^+}\}$. By \eqref{2.1}, we
have $\overline{\rho_g^+}=(-1)^{w+1}\rho_{g^*}^+$, $\overline{\rho_g^-}=(-1)^w\rho_{g^*}^-$,
and the previous
equality reduces to $\{\rho_f , \rho_{g^*}\}=0$, which is Theorem \ref{thm_hab} b).
\end{proof}

\section{Coboundary polynomials}\label{sec_cw}

In this section we show that the space of coboundary polynomials 
$$C_w^\Gamma=\{P|(1-S): P \in V_w^\Gamma \cap \ker(1-T) \}
$$ is the radical
of the bilinear form $\{\cdot, \cdot\}$ on $W_w^\Gamma$. The dimension of $C_w^\Gamma$ equals
the dimension of the Eisenstein subspace $\EE_k(\Gamma)\subset M_k(\Gamma)$. For
$\Gamma=\Gamma_0(N)$, we characterize those $N$ for which
$(C_w^\Gamma)^-$ is trivial, namely those $N$ for which the map $\rho^-:S_k(\Gamma)\rightarrow (W_w^\Gamma)^-$ is
an isomorphism, as in the full level case.

\begin{lemma}\label{l4.4} Let $\Gamma$ be a finite index subgroup of $\Gamma_1$. The period
polynomials $W_w^\Gamma$ are
orthogonal to the coboundary polynomials $C_w^\Gamma\subset W_w^\Gamma$ with respect to the
pairing $\{\cdot , \cdot \}$.
\end{lemma}
\begin{proof}
 Let $P|(1-S)\in C_w^\Gamma$ with $P|1-T=0$, and let $Q\in W_w^\Gamma$. Then $\lla
P|(1-S)(T-T^{-1}),Q \rra=0$ follows from the following
relation in $\Z[\og_1]$
\begin{equation}\label{4.8}
(1-S)(T-T^{-1})=(T-1)(2+T^{-1}+ST-S)+(1+TS+ST^{-1})(1-S)
\end{equation}
using the $\Gamma_1$ invariance of the pairing $\lla\cdot , \cdot \rra$ (recall $U=TS$). 
\end{proof}

Let $e_\infty(\Gamma)$, $e_\infty^\reg(\Gamma)$ denote the number of inequivalent cusps,
respectively regular cusps \cite[Ch. 3]{DS}.  The next lemma shows that $\dim
C_w^\Gamma=\dim \EE_k(\Gamma)$.
\begin{lemma} \label{L7.1} 
Let $\Gamma$ be any finite index subgroup of $\Gamma_1$. The dimension of $C_w^\Gamma$ equals:
$e_\infty(\Gamma)$ if $k>2$ is even; $e_\infty(\Gamma)$-1 if $k=2$; $e_\infty^\reg(\Gamma)$ if
$k>2$ is odd.
\end{lemma}
\begin{proof} 
 Let $P\in V_w^\Gamma \cap \ker(1-T)$.
Then $P|T^n= P$ for every $n\in \Z$, that is $P(A T^{-n})|T^n =P(A)$ for $A\in
\Gamma\backslash\Gamma_1$. Since $\Gamma\backslash\Gamma_1$ is finite, there is $n$ such
that $AT^{-n}=A$, and $P(A)(X+n)=P(A)(X)$. Since the only periodic polynomials are the
constants, it follows that $P(A)(X)=c_A\in \C$, with $c_{AT}=c_A$. Since $P|J=P$ we also
have $c_{AJ}=(-1)^w c_A$. 
Hence we have:
\begin{equation}\label{cw}
C_w^\Gamma=\{(c_A-c_{AS^{-1}}X^w)_A\in V_w^\Gamma: c_A\in \C, \ c_{AT}=c_A, \ c_{AJ}=(-1)^w c_A
\}.
 \end{equation}  

If $k>2$ is even it follows that $ \dim C_w^\Gamma=|\Gamma\backslash \Gamma_1/
\Gamma_{1\infty}|$,
with $\Gamma_{1\infty}=\{\pm T^n: n\in \Z\}$ the stabilizer of $\infty$. Since the map 
\[
\Gamma\backslash \Gamma_1/ \Gamma_{1\infty}\rightarrow  \Gamma \backslash \mathbb{P}^1(\Q),
\quad [\gamma]\rightarrow [\gamma\infty]
\]
is a bijection and  $|\Gamma \backslash \mathbb{P}^1(\Q)|=e_\infty(\Gamma)$, the claim
follows.

If $k=2$, we identify $\C^{e_\infty(\Gamma)}$ with the vector space $\{(c_A)_{A\in
\Gamma\backslash \Gamma_1}: c_{AT}=c_A=c_{AJ}\in\C \}$ and define the (surjective) map
$\C^{e_\infty(\Gamma)}\rightarrow C_w^\Gamma$ by $(c_A)_A \rightarrow (c_A-c_{AS})_A $.  Its
kernel consists of those vectors $(c_A)_A$ with $c_{AT}=c_{AS}=c_{AJ}=c_{A}$. Since $S,T,J$
generate $\Gamma_1$, it follows that $c_A=c$ for all $A\in \Gamma\backslash\Gamma_1$, so the
kernel is isomorphic to $\C$. Therefore  $ \dim C_w^\Gamma=e_\infty(\Gamma)-1$.

If $k>2$ is odd (so $-1\notin \Gamma$), we have $c_{AJ}=-c_A$. Therefore $\dim C_w^\Gamma$
equals the number of classes $[A] \in \Gamma\backslash \Gamma_1/\Gamma_{1\infty}$, such that
the two associated classes $[A]^+, [AJ]^+\in \Gamma\backslash\Gamma_1/\Gamma_{1\infty}^+$ are
distinct, where $\Gamma_{1\infty}^+=\{T^n:n\in \Z\}$ (when $[A]^+=[AJ]^+$ then clearly
$c_A=c_{AJ}=0$). But $[A]^+\ne[AJ]^+$ precisely when $[A]$ corresponds to a regular cusp of
$\Gamma$ since $[A]^+=[AJ]^+$ means that $A^{-1}\gamma A=-T^n$ for some $\gamma\in \Gamma$,
$n>1$, so the cusp $A\infty$ is irregular. We conclude that $\dim
C_w^\Gamma=e_\infty^\reg(\Gamma)$.
\end{proof}

\begin{definition}\label{d_cusp}
Following the proof of Lemma \ref{L7.1}, we call \emph{cusp} a double coset $\cc\in
\Gamma\backslash
\Gamma_1/\Gamma_{1\infty}$, which corresponds to the $\Gamma$-equivalence class of the usual
cusp $A\infty$ for any representative $A\in\cc$. We call \emph{regular} those
cusps $\cc=[A]$ such that the double cosets $[A]^+, [AJ]^+ \in
\Gamma\backslash\Gamma_1/\Gamma_{1\infty}^+$ are distinct. The terminology agrees with the
usual one for the cusp $A\infty$ by the last paragraph in the proof of the lemma.
\end{definition}

For $\Gamma_0(N)$ it turns out that $(C_w^\Gamma)^-$ is often
trivial, in which case $\rho^-:S_k(\Gamma)\rightarrow (W_w^\Gamma)^-$ is an isomorphism just like for $\Gamma_1$.
The following proposition was discovered using SAGE \cite{Sg}.
\begin{prop}\label{prop7.2} Let $\Gamma=\Gamma_0(N)$. Then $(C_w^\Gamma)^-=\{0\}$
if and
only if $N=2^e N'$ with $N'$ odd square free and $0\le e\le 3$.   
\end{prop}
\begin{proof} From the proof of Lemma \ref{L7.1} we identify $(C_w^\Gamma)^-$ with the
space $(\C^{e_\infty(\Gamma)})^-$ of vectors $(c_A)_{A\in\og\backslash\og_1}$ with $c_A=c_{AT}$
and $c_A=-c_{A'}$ (including for $k=2$). 

Assume $N$ does not satisfy the conditions, so there exists $t\ge 3$ with
$t^2|N$. We claim that $[A]\ne [A']$ for $A= \left(\begin{smallmatrix} x & y \\
t & z \end{smallmatrix}\right)\in \Gamma_1$, so $(C_w^\Gamma)^-\ne\{0\}$. Assuming by
contradiction that $\gamma A T^s=A'$ for $\gamma=\left(\begin{smallmatrix} * & * \\
c & d \end{smallmatrix}\right)\in \Gamma$, it follows that $cx+dt=-t, c(y+sx)+d(z+st)=z$. The
first equation implies that $t|d+1$ while the second that $t|d-1$, a contradiction with $t\ge
3$. 

Assuming $N$ satisfies the conditions, let $(c_A)_{A\in\og\backslash\og_1}$ with $c_A=c_{AT}$
and $c_A=-c_{A'}$. Identifying the coset space $\Gamma_0(N)\backslash\Gamma_1$ with
$\mathbb{P}^1(\Z/N\Z)$, it follows that 
$c_{(a:b)}=c_{(a:a+b)}, \ c_{(a:b)}=-c_{(-a:b)}.
$ 
The second relation implies $c_{(0:1)}=0$. Let $N=d d'$ and $k\in \Z$, $(k,d)=1$. We will
show that $c_{(d:k)}=0$. We have:
\begin{equation*}
c_{(d:k)} =c_{(d:k+ad)} = -c_{(-d:k)}=-c_{((bd'-1)d: k)} 
\end{equation*}
and it is enough to find $ a,b\in \Z$ with $(bd'-1,N)=1$ and $k\equiv (bd'-1)(k+ad)
\pmod{N}$. The latter equation can be written  
$ k(bd'-2)\equiv ad \pmod{d d'}$
and the hypothesis on $N$ ensures that $(d,d')|2$, so that we can find $b,u$ such that
$bd'-2=du$. Taking $a\equiv ku \pmod{d'}$, it follows that $(d:k+ad)=((bd'-1)d:k)$, which
implies that $c_{(d:k)}=0$. It follows that $c_A=0$ for all $A\in \og\backslash\og_1$,
finishing the proof.
\end{proof}

\section{Hecke operators}\label{sec4}

Following the Eichler integral method sketched in \cite{Za93} for the modular group, we show that the same
elements $\wT_n$ as in the full level case define actions on period polynomials corresponding to a large class of
double coset operators on modular forms, including the Hecke operators and Atkin-Lehner operators. On modular
symbols, the action of the same type of double coset operators was determined by Merel, and the results of
$\S$\ref{sec5.11} parallel those of \cite{Me}. We find the adjoints of these operators with respect to the pairing
$\{\cdot ,\cdot\}$ in $\S$\ref{sec5.12}.\footnote{For $\Gamma_1$, the Hecke equivariance of the pairing is
mentioned without proof in \cite[p.96]{GKZ}.} In $\S$\ref{sec5.1} we give a new proof of the rationality of the
plus and minus parts of period polynomials of newforms, while in $\S$\ref{sec5.2} we discuss modular forms with
Nebentypus. We end with a discussion of the numerical computation of period polynomials,
Hecke eigenvalues, and Petersson norms of newforms. 
 
\subsection{The universal Hecke operators}\label{sec5.11}

Let $M_n$ be the set of integer matrices of determinant $n$, set  $\ov{M}_n=M_n/\{\pm I\}$, and let 
$R_n=\Q[\ov{M}_n]$. Thus $\og_1$ acts on $R_n$ by left and right multiplication. Let
$$M_n^\infty=\big\{\left(\begin{smallmatrix} a & b \\
0 & d \end{smallmatrix}\right):n=ad, 0\le b<d\big\}$$ be the usual system of representatives for
$\Gamma_1\backslash M_n$, and $T_n^\infty=\sum_{M\in M_n^\infty} M \in R_n$. Following 
\cite{CZ}, let  $\wT_n, Y_n\in R_n$ be such that
\begin{equation}\label{hecke}
T_n^{\infty}(1-S)=(1-S)\wT_n+(1-T)Y_n . 
\end{equation}
We will show that for all $n$, the elements $\wT_n$ define Hecke operators on period
polynomials for a large class of congruence subgroups. Note that the elements $\wT_n$ are universal, not depending
on the weight or level. 

Let $\Gamma$ be a finite index subgroup of $\Gamma_1$ normalized by $\epsilon$, which we will often specialize to
be $\Gamma_1(N)$ or $\Gamma_0(N)$. Fix an integer $n\ge 1$, and let $\Sigma_n\subset M_n$ such that
$\Sigma_n=\Gamma \Sigma_n =\Sigma_n \Gamma$ and $\Sigma_n$
is a disjoint finite union of cosets $\Gamma \sigma$. Such double cosets $\Sigma_n$ define 
operators $[\Sigma_n]$ on $f\in M_k(\Gamma)$ 
\[
f|[\Sigma_n]=n^{w+1} \sum_{\sigma \in \Gamma \backslash \Sigma_n} f|_k \sigma. 
\]
Setting $g^\vee=g^{-1}\det g$, the adjoint of $[\Sigma_n]$ with respect to the Petersson product is given by 
$\Sigma_n^\vee=\{g^\vee: g\in \Sigma_n\}$, namely $(f|[\Sigma_n], g)=(f,g|[\Sigma_n^\vee] )$.

We now define an action of $M_n$ on $V_w^\Gamma$ which depends on $\Sigma_n$, and which is based on
the following property of the pair $(\Gamma,\Sigma_n)$
 \begin{equation} \label{eq_star} \tag{H} 
\text{The map $\Gamma\backslash
\Sigma_n\rightarrow  \Gamma_1\backslash\Gamma_1\Sigma_n, \ \ \ $  $\Gamma \sigma \mapsto \Gamma_1 \sigma $ 
is bijective.}
\end{equation}
See \cite[Prop. 3.36]{Sh} for a large class of congruence subgroups when this property holds.

For $M\in M_n$, $A\in \Gamma_1$ and $M A^{-1}\in \Gamma_1 \Sigma_n 
$, there exists a decomposition $M A^{-1}=A_M^{-1} M_A$, with
$M_A\in \Sigma_n$, $A_M\in \Gamma_1$, and by \eqref{eq_star} the coset $\Gamma A_M$ is independent of the
decomposition. Moreover $\Gamma A_M$ depends only on $\Gamma A$, since $\Gamma \Sigma_n=\Sigma_n \Gamma$.
For $P\in V_w^\Gamma$ we define an element $P|_{\Sigma}M =P|_{\Sigma_n}M\in V_w^\Gamma$
by\footnote{With the notation $P|_{\Sigma}M$, the dependence on $n$ is recorded in the fact that $M\in M_n$.}
\begin{equation}\label{eq_act}
P|_{\Sigma}M(A)=\begin{cases} P(A_M)|_{-w}M  & \text{ if } M A^{-1} \in \Gamma_1 \Sigma_n \\
          0 & \text{ otherwise.}
         \end{cases}
\end{equation}
Note that both $M$ and $JM$ act in the same way (recall $J=-I$), 
so the action of $\ov{M}_n=M_n/ \{\pm I\}$ is also well defined. This action extends linearly to an action of
$R_n$ on $V_w^\Gamma$. In the same way we define an action of $\ov{M}_n$ and $R_n$ on the Eichler integrals $\tf$
in \eqref{eq_eich}.

We are interested in the following double cosets $\Sigma_n$ satisfying \eqref{eq_star}, for $\Gamma=
\Gamma_1(N)$ or $\Gamma_0(N)$. 
\begin{enumerate}
 \item[(1)] The double coset $\Delta_n$ consisting of $\left(\begin{smallmatrix} a & b \\
c & d \end{smallmatrix}\right)\in M_n$ with  $N|c$, and $a\equiv 1 \pmod{N}$ if $\Gamma=\Gamma_1(N)$, or $(a,N)=1$
 if $\Gamma=\Gamma_0(N)$. The operator $[\Delta_n]$ is the usual
Hecke operator, denoted by $T_n$, and  Property \eqref{eq_star} follows from \cite[Prop. 3.36]{Sh}. Note that for
$(n,N)=1$ we have $\Gamma_1 \Delta_n=M_n$, so the second case in \eqref{eq_act} does not occur.  

\item[(2)] The double coset $\Delta_n^\vee$ for $(n,N)=1$. The operator $[\Delta_n^\vee]$ is denoted by $T_n^*$,
the adjoint of
$T_n$ with respect to the Petersson inner product. Property \eqref{eq_star} can be checked directly. 

\item[(3)] The double coset $\Theta_n=\Gamma w_n \Gamma $ for $N=n n'$, $(n,n')=1$, where
$w_n=\left(\begin{smallmatrix} nx & y \\
Nz & nt \end{smallmatrix}\right)\in M_n$ with $x,y,z,t\in \Z$ for $\Gamma=\Gamma_0(N)$, while for 
$\Gamma=\Gamma_1(N)$ we impose the extra conditions $nx\equiv 1 \pmod{n'}$, $y\equiv 1 \pmod{n}.$  Then
$\Gamma\backslash \Theta_n$ has one element so Property \eqref{eq_star} is trivially satisfied. The operator
$[\Theta_n]$ is denoted by $W_n$, and if $W_n^*=[\Theta_n^\vee]$ denotes its adjoint we have $f|W_n|W_n^*=n^w f$.
For $\Gamma=\Gamma_0(N)$ (so that $W_n=W_n^*$),  the usual Atkin-Lehner
involution is $W_n/n^{w/2}$ due to our choice of normalization. 

\item[(4)] The double coset $\Theta_n^\vee$ with $\Theta_n$ as in (3). If $\Gamma=\Gamma_0(N)$ then
$[\Theta_n^\vee]=[\Theta_n]$,
and if
$\Gamma=\Gamma_1(N)$ then $[\Theta_n^\vee]=\la d \ra[\Theta_n]$ for $d\equiv -1
\pmod{n}$, $nd\equiv 1 \pmod{n'}$, where $\la d\ra$ denotes the diamond operator.
\end{enumerate}

\begin{remark}\label{r5.1} Assume $\Gamma=\Gamma_1(N)$ and $\Sigma_n$ is as in (1) or (3). The action of $M_n$
on $V_w^\Gamma$ can be determined as follows. Applying adjoint we have that
$MA^{-1}\in\Gamma_1 \Sigma_n$ if and only if
$$
\left(\begin{smallmatrix} a & b \\
c & d \end{smallmatrix}\right)=AM^\vee=M_A^\vee A_M  \in \Sigma_n^\vee\Gamma_1.
$$ 

For $\Sigma_n=\Delta_n$, the inclusion is satisfied if and only if $\gcd(c,d,N)=1$, and in that case
$A_M= \left(\begin{smallmatrix} * & * \\
c' & d' \end{smallmatrix}\right)\in \Gamma_1$ has $c'\equiv c, d'\equiv d\pmod{N}$, which
uniquely defines its class in $\Gamma\backslash \Gamma_1$. 

For $\Sigma_n=\Theta_n$, let $N=n n'$ with $n\|N$. The inclusion is satisfied if and only if $n|c, n|d$, and
the class of $A_M=\left(\begin{smallmatrix} * & * \\
c' & d' \end{smallmatrix}\right)$ in $\Gamma\backslash \Gamma_1$ is determined by
\[c'\equiv -a ,\ d'\equiv -b \pmod{n}, \quad c'\equiv c/n ,\ d'\equiv d/n \pmod{n'}.  
\]
When $\Gamma=\Gamma_0(N)$ and $\Sigma_n=\Theta_n$, the last congruences are replaced by $yc'\equiv a, yd'\equiv b
\pmod{n}$ and $tc'\equiv c/n ,\ td'\equiv d/n \pmod{n'}$ where $t,y$ are any integers such that $nxt-n'yz=1$
(taking $M_A^\vee=\left(\begin{smallmatrix} nx & y \\
Nz & nt \end{smallmatrix}\right)\in M_n$). The class of $(c',d')$ in $P^1(\Z/N\Z)\sim \Gamma\backslash\Gamma_1$ is
then independent of $y,t$. 
\end{remark}

While the operation $P|_{\Sigma}M$ is not a proper action, it is compatible with the action of $\Gamma_1$ and of
$\epsilon$. For
$h\in \Gamma_1$, $M\in M_n$ a formal computation shows that
\begin{equation}\label{Ecom} 
P|_{\Sigma}M|h=P|_{\Sigma}Mh, \ \ P|h|_{\Sigma} M=P|_{\Sigma}hM, \ \ 
P|_{\Sigma}M|\epsilon=P|\epsilon|_{\Sigma'}\epsilon M\epsilon,
\end{equation}
where $\Sigma_n'=\{g'=\epsilon g\epsilon : g\in \Sigma_n \}$. Using the compatibility with the $\Gamma_1$ action,
we
show that the universal operators $\wT_n$ play on period polynomials the role  of \emph{all} double coset operators
$[\Sigma_n]$ on modular forms, as long as the pair $(\Gamma, \Sigma_n)$ satisfies \eqref{eq_star}.

\begin{prop} \label{p4.2}
Assume the pair $(\Gamma, \Sigma_n)$ satisfies \eqref{eq_star}. If 
$\wT_n\in R_n$ satisfies \eqref{hecke} then for every $f\in S_k(\Gamma)$
$$\rho_{f|[\Sigma_n]}=\rho_f |_{\Sigma} \wT_n.$$
\noindent In particular for $\Gamma=\Gamma_1(N)$, we have $\rho_{f|T_n}=\rho_f |_{\Delta} \wT_n$ and, if $n\|N$,
$\rho_{f|W_n}=\rho_f |_{\Theta} \wT_n$.
\end{prop} 
\begin{proof} Using the fact that $\rho_f=\tf|(1-S)$ 
and $\tilde{f}|(1-T)=\sigma_f(T)=0$, we have as in \cite{Za93} 
\[
\rho_f|_{\Sigma}\wT_n = \tf|_{\Sigma}(1-S)\wT_n=
\tilde{f}|_{\Sigma}T_n^{\infty}(1-S)=\widetilde{f|[\Sigma_n]}|(1-S)=\rho_{f|[\Sigma_n]},
\]
once we show that $\tilde{f}|_{\Sigma}T_n^{\infty}=  \widetilde{f| [\Sigma]}$. 

For $M\in M_n^\infty$ and  $A\in\Gamma_1$, let $M A^{-1}=A_M^{-1} M_{A}$ with
$A_M\in\Gamma_1$, $M_{A}\in \Sigma_n$. By \eqref{eq_act}
\begin{equation}\label{4.3}
\begin{split}
\tf|_{\Sigma}T_n^{\infty}(A)&=\sum_{M\in M_n^\infty\cap \Gamma_1 \Sigma_n A }\int_{M z}^{i \infty} f|A_M
(t)(t-Mz)^w j(M,z)^w dt\\
&=n^{w+1} \sum_{M\in M_n^\infty \cap \Gamma_1 \Sigma_n A}\int_{z}^{i \infty}f|M_{A} A
(u)(u-z)^w du
\end{split}
\end{equation}
where we made a change of variables $t=M u$. For fixed $A$, the map 
$$ M_n^\infty\cap \Gamma_1 \Sigma_n A\rightarrow \Gamma\backslash\Sigma_n,  \quad M\mapsto M_{A}$$
is well defined by \eqref{eq_star}. It is easy to check that the map is bijective,
hence the last expression equals $\widetilde{f|[\Sigma_n]} (A)$ finishing the proof.
\end{proof}
\begin{corollary}\label{c4.3} If in addition to the hypotheses of Proposition \ref{p4.2} we assume that $\Gamma$ is
normalized by $\epsilon$ and that $\Sigma_n'=\Sigma_n$ then  $$\rho_{f|[\Sigma_n]}^\pm=\rho_f^\pm |_{\Sigma}
\wT_n.$$ 
In particular if $\Gamma=\Gamma_1(N)$, then
$\rho_{f|T_n}^\pm=\rho_f^\pm |_{\Delta} \wT_n$, and if $\Gamma=\Gamma_0(N)$, then
$\rho_{f|W_n}^\pm=\rho_f^\pm |_{\Theta} \wT_n$.
\end{corollary}
\begin{proof}  By the last compatibility in \eqref{Ecom}, it is enough to show that $\epsilon \wT_n \epsilon$ also
satisfied \eqref{hecke} (for a different $Y_n$).
This follows from conjugating \eqref{hecke} by $\epsilon$, and using that
$T_n^{\infty}-\epsilon T_n^{\infty}\epsilon \in (1-T) R_n$. 

We have $\Delta_n'=\Delta_n$ and, for $\Gamma=\Gamma_0(N)$, $\Theta_n^{\prime}=\Theta_n$. Note that  if
$\Gamma=\Gamma_1(N)$, then $[\Theta_n^{\prime}]= \la d \ra [\Theta_n]$ with $d\equiv -1\pmod{n}$, $d\equiv
1\pmod{n'}$, so the action $|_{\Theta} \wT_n$ on $(W_w^\Gamma)^\pm$ does not correspond to $W_n$ in this
situation.
\end{proof}

Taking $\Sigma_n=\Delta_n$ in Corollary \ref{c4.3}, we obtain explicit formulas for the Fourier coefficients of a
Hecke eigenform $f\in S_k(\Gamma)$ in terms of the polynomials $\rho_f^\pm$, generalizing the Coefficients Theorem
of
Manin \cite{M73}. The formulas can be used for fast computation of Hecke eigenvalues, as explained in
$\S$\ref{sec5.3}, and they also yield an explicit inverse of the Eichler-Shimura maps $\rho^\pm$ of Theorem
\ref{thm2.1}.

We state the formula corresponding to $\rho_f^+$. For $\Gamma=\Gamma_1(N)$ we identify a coset representative
$A\in \Gamma\backslash \Gamma_1$ with the element $ (c_A,d_A)\in E_N$, where $c_A$, $d_A$ are the lower left,
respectively lower right
entries of $A$ and  
\[
E_N=\big\{(c,d) \in  (\Z/N\Z)^2 : (c,d,N)=1\big\}.
\]
For $P\in W_w^\Gamma$, we write $P(A)=P(c_A,d_A)\in V_w$. 
\begin{prop}\label{p5.4}
Let $f\in S_k(\Gamma_1(N))$ be an eigenform for $T_n$ with eigenvalue $\lambda_n$, and let $\wT_n=\sum_{M\in M_n}
\alpha(M) M$ be any operator satisfying \eqref{hecke}.  

$\mathrm{(a)}$ Assume $r_{I,w}(f)\ne 0$ (which is satisfied if $k>2$ and $f$ is a newform), and let $P_f^+
=\rho_f^+/r_{I,w}(f)$. Then  
\[
\lambda_n=\sum_{\substack{M\in M_n\\(c_M,a_M,N)=1}} \alpha(M) P_f^+(-c_M, a_M)|M(0) 
\]
where $M=\left(\begin{smallmatrix} a_M & b_M \\
c_M & d_M \end{smallmatrix}\right) $. When $(n,N)=1$ the congruence condition in the sum can be omitted. 

$\mathrm{(b)}$ If $k=2$, let $(x,y)\in E_N$ represent a coset in $ \Gamma\backslash \Gamma_1$ with $\rho_f^+(x,y)
\ne 0$ and let $P_f^+=\rho_f^+/\rho_f^+(x,y)$. For each $M\in M_n$, let $x_M=xd_M-yc_M$, $y_M=-xb_M+ya_M$. Then
\[
\lambda_n=\sum_{\substack{M\in M_n\\(x_M,y_M,N)=1}} \alpha(M) P_f^+(x_M, y_M). 
\]
\end{prop}
\begin{remark}
By Proposition \ref{p5.7}, when $f$ is a newform the polynomials $P_f^+$ have coefficients in the field $K_f$ of
coefficients of $f$. When $f$ has known Nebentypus, we obtain simpler
formulas using the results of $\S$\ref{sec5.2}. 
\end{remark}
\begin{proof}
By Corollary \ref{c4.3}, $P_f^+$ is an eigenform of $\wT_n$ with eigenvalue $\lambda_n$, and the conclusion
follows by writing the action of $\wT_n$ on $P_f^+$ explicitly, using Remark \ref{r5.1}.  
\end{proof}

Elements $\wT_n$ satisfying condition \eqref{hecke} go back to work of Manin \cite{M73},
and particular examples are given in \cite{CZ,Z90}. The element
$\wT_n$ is unique, up to addition of any element in the right $\Gamma_1$-module 
\begin{equation}\label{ideal}
 \cI=(1+S)R_n+(1+U+U^2)R_n.
\end{equation}\vspace{-5mm}
\begin{remark}\label{r5.4} In \cite{Me} Merel gives several examples of elements $\widetilde{\mathfrak{T}}_n\in
R_n$ acting on modular symbols, satisfying a condition that plays the same role as \eqref{hecke}. It can
be shown that the elements
$\widetilde{\mathfrak{T}}_n^\vee$ (with the notation explained in the next paragraph) satisfy Property
\eqref{hecke}, reflecting the fact that the action of $\Gamma_1$ on modular symbols is on the left, and on period
polynomials is on the right. 
\end{remark}

\subsection{Adjoints of Hecke operators} \label{sec5.12}
Next we determine the adjoint of the action of Hecke and Atkin-Lehner operators for the pairing on
$W_w^\Gamma$ defined in \eqref{pairing}. For $g\in M_n$ we denote by $g^\vee=g^{-1}\det g $
the adjoint of $g$, and we apply this notation to all elements of $R_n$ by linearity. 
Recall that $\la P|g,Q \ra=\la P,Q|g^\vee\ra$ for $P,Q\in V_w$, $g\in GL_2(\C)$. 

\begin{lemma} Assume that both $(\Gamma,\Sigma_n)$ and $(\Gamma,\Sigma_n^\vee)$
satisfy \eqref{eq_star}. For $P,Q\in V_w^\Gamma, M\in M_n$ we have
\begin{equation}\label{3.4}
 \lla P|_{\Sigma} M,Q \rra=\lla P,Q|_{\Sigma^\vee} M^\vee\rra .
\end{equation}
In particular, \eqref{3.4} holds for $\Gamma=\Gamma_1(N)$ and $\Sigma_n=\Delta_n$ with $(n,N)=1$, or
$\Sigma_n=\Theta_n$ for $n\|N$. 
\end{lemma}
\begin{proof} For $A\in \Gamma_1$ and $M\in \Gamma_1 \Sigma_n A$, let $M A^{-1}=A_M^{-1}
M_A$ with $A_M\in\Gamma_1, M_A\in \Sigma_n$. We have (see Remark \ref{r_sign})
 \[
[\Gamma_1:\Gamma] \lla P|M,Q \rra=\sum_{\substack{A\in \Gamma\backslash \Gamma_1\\ 
MA^{-1}\in \Gamma_1 \Sigma_n}} \la P(A_M), Q(A)|M^\vee \ra. 
\]
Taking adjoint we have
$M^\vee A_M^{-1}=A^{-1}M_A^\vee$ with $M_A^\vee\in \Sigma_n^\vee$, hence $MA^{-1}\in \Gamma_1 \Sigma_n$ if and
only if $M^\vee A_M^{-1}\in \Gamma_1 \Sigma_n^\vee$. Moreover  the map $A\mapsto A_M$ is injective, by Property
\eqref{eq_star} applied to $\Sigma_n^\vee$. Summing over $A_M$ instead of $A$ finishes the
proof. 
\end{proof}

Next we give two proofs of the Hecke equivariance of the period polynomial pairing, one of
them requiring the following lemma.

\begin{lemma}\label{l4.40}
The space $C_w^\Gamma$ is preserved by the Hecke operators $\wT_n$, whenever an 
action of $M_n$ on $V_w^\Gamma$ satisfying \eqref{Ecom} can be defined.  
\end{lemma}
\begin{proof}
Let $P|(1-S)\in C_w^\Gamma$ with $P|1-T=0$. By \eqref{hecke}, 
$P(1-S)|_{\Sigma}\wT_n=P|_{\Sigma}T_n^\infty (1-S)$, and the latter is an element of
$C_w^\Gamma$, since $T_n^\infty(1-T)\in (1-T)R_n$, so $P|_{\Sigma}T_n^\infty(1-T)=0$. 
\end{proof}

\begin{theorem} \label{thm_equiv} Assume that both $(\Gamma,\Sigma_n)$ and $(\Gamma,\Sigma_n^\vee)$
satisfy property \eqref{eq_star}. For $P,Q\in W_w^\Gamma$ and any $\wT_n$ as
in \eqref{hecke} we have:
\[
\{ P|_{\Sigma}\wT_n, Q \} =\{P, Q|_{\Sigma^\vee} \wT_n\}.
\]
\end{theorem}
\begin{proof} We give two proofs. For the first we assume that both $\Sigma_n, \Sigma_n^\vee$ satisfy
\eqref{eq_star}, and in addition $\Sigma_n=\Sigma_n'$. From Theorem
\ref{thm_main} and Corollary \ref{c4.3}, it follows that the claim is true for $P=\rho_f^{\pm}$ and
$Q=\rho_g^\mp$, for any $f,g\in S_k(\Gamma)$. By \eqref{7.1}, any $P\in W_w^\Gamma$ can be written $P=\rho_f^+
+\rho_g^- + Q$ with $Q\in C_w^\Gamma$. Since $\wT_n$ preserves $(W_w^\Gamma)^\pm$ and $C_w^\Gamma$, the claim
follows taking into account Lemmas \ref{l4.4}.  

The second proof is purely algebraic and we only assume that $\Sigma_n, \Sigma_n^\vee$ satisfy
\eqref{eq_star}. Via \eqref{3.4} the equality to prove is equivalent to
\[\lla P|_{\Sigma}\big[\wT_n(T-T^{-1})+(T^{-1}-T)\wT_n^\vee\big], Q\rra=0.
\] 
Since $P,Q\in W_w^\Gamma$, we are reduced to proving the next theorem. 
\end{proof}

\begin{theorem}\label{thm_hecke}
For any element $\wT_n\in R_n$ satisfying property \eqref{hecke} we have
\[
\wT_n(T-T^{-1})+(T^{-1}-T)\wT_n^\vee \in \cI+\cI^\vee,
\]
where $\cI$ is defined in \eqref{ideal}.
\end{theorem}
\noindent The proof is quite involved, and we give it in the short article \cite{PP12a}.

\subsection{Rationality of period polynomials of Hecke eigenforms}\label{sec5.1}

Let $\Gamma=\Gamma_1(N)$ and for a character $\chi$ modulo $N$ let $S_k(N,\chi)\subset S_k(\Gamma)$ be the
subspace of forms of Nebentypus $\chi$. The
following proposition is well-known, although the precise statement is hard to find in the literature. For the
trivial character, an equivalent statement to part (a) was given in terms of modular symbols in \cite{GS}. The
proof requires the extension of the pairing $\{ \cdot ,\cdot\}$ to the whole space of modular forms and its
properties proved in Section \ref{sec8}.

\begin{prop} \label{p5.7} Assume that $f\in S_k(N, \chi)$ is a newform (a normalized eigenform for all Hecke
operators which does not come from lower levels) and let $K_f$ be the
field of coefficients of $f$. 

$\mathrm{(a)}$ There exist nonzero complex numbers $\om_f^+$ and $\om_f^{-}$ such that all the polynomial
components of $\rho_f^\pm/\omega_f^{\pm}$ have coefficients in $K_f$. 

$\mathrm{(b)}$ We have 
\[ \frac{\omega_f^+ \overline{\omega_f^-}}{i(2\pi)^{k-1} (f,f)} \in K_f \text{ if $k$ is even}, \quad \quad 
\frac{|\omega_f^\pm|^2 }{(2\pi)^{k-1} (f,f)} \in K_f \text{ if $k$ is odd}. 
\]
In particular, if $k$ is even one can choose $\omega_f^\pm$ such that $\om_f^+\overline{\om_f^-}= i (2\pi )^{k-1}
(f,f)$.

$\mathrm{(c)}$ If $f$ has real Fourier coefficients at infinity, then $\om_f^+\in i^{k+1} \R $ and $\om_f^-\in
i^k\R$. 
\end{prop}

\begin{proof} (a) We view $f\in S_k(\Gamma)$ for $\Gamma=\Gamma_1(N)$. By multiplicity
one and Corollary \ref{c4.3}, the polynomials $\rho_f^\pm$ are the unique (up to scalars) elements in
$\rho^\pm(S_k(\Gamma))\subset W_w^\Gamma$ with the same eigenvalues under $|_\Delta\wT_n$ as those of $f$.
Moreover on $C_w^\Gamma$ there is no eigenvector for $|_{\Delta}\wT_n$ with the same eigenvalues $\lambda_n$ as
those of $f$ for every $n$ coprime to $N$. Indeed, $|\lambda_n|\ll n^{(k-1)/2+\epsilon} $ by the Ramanujan bounds
(even a nontrivial upper bound would suffice), and by Corollary \ref{c8.9} the eigenvalues of $|_\Delta\wT_n$ on
$C_w^\Gamma$ are the same as the eigenvalues of an Eisenstein series, which are of order $n^{k-1}$ if $(n,N)=1$.

Therefore the common eigenspace of $|_{\Delta}\wT_n$ for $(n,N)=1$ acting on $(W_w^\Gamma)^\pm$ having the same
eigenvalues as those of $f$  is one dimensional, generated by $\rho_f^\pm$.  Since the matrix of $\wT_n$ has
integer coefficients with respect to a basis of $V_w^\Gamma$, and $(W_w^\Gamma)^\pm$ are subspaces cut by relations
with integer coefficients, it follows that there is a common eigenvector defined over $K_f$.

(b) The claim follows from (a) and Theorem \ref{thm_main}.  

(c) The conclusion follows from \eqref{2.1}. 
\end{proof}

The coefficients of $\rho_f^\pm(I)$ are related to the critical values of the $L$-function of $L(s,f)$, so the
proposition implies well-known rationality statements about critical values. With the notation of 
\eqref{5.1}, we have
\begin{equation}\label{5.7}
r_{I,n}(f)=i^{n+1}\Lambda(n+1,f),  \text{ with } \Lambda(s,f):=(2\pi)^{-s}\Gamma(s)L(s,f)  
\end{equation}
the completed $L$-function. This proves the following special case of Theorem 1 in \cite{Sh77}. 

\begin{corollary} Assume that $f\in S_k(N, \chi)$ is a newform, and let $K_f$ be the field of
coefficients of $f$. Then with the periods $\om_f^\pm$ defined in the
Proposition \ref{p5.7} we have 
\[ \frac{i^{n}\Lambda(n,f)}{\om_f^\pm}\in  K_f  \text{ for } 0<n<k, \quad (-1)^{n}=\pm (-1)^{k-1}.
\]
\end{corollary}
\noindent When $k$ is even, our convention regarding the signs of the transcendental factors differs from
the usual one eg. in \cite{Sh77}, since $\omega_f^+$ is the normalizing factor for \emph{odd} critical values (and
for \emph{even} period polynomials). 
\begin{remark}
When $k$ is odd it follows from Prop. \ref{p5.7} (b) that $\frac{|\omega_f^+|^2}{|\omega_f^-|^2}\in K_f$. In fact,
when $K_f$ is a real field, it follows from the
functional equation $N^{s/2}\Lambda(s,f)=\pm N^{(k-s)/2} \Lambda(k-s,f)$ that $\frac{\omega_f^+}{i\omega_f^-}\in
K_f(\sqrt{N})$. From \eqref{5.7} and the functional equation, when $K_f$ is real we can choose 
$$\omega_f^-=-i \Lambda(1,f),\quad \omega_f^+=\pm i^{w+1} N^{(k-1)/2}\Lambda(k-1,f)$$ 
so that $\frac{\omega_f^+}{i\omega_f^-}=\sqrt{N}$. 
\end{remark}

\subsection{Modular forms with Nebentypus} \label{sec5.2}

Let $\Gamma=\Gamma_1(N)$, $\Gamma'=\Gamma_0(N)$ and $\chi$ a Dirichlet character modulo $N$. For $f\in S_k(N,
\chi)\subset S_k(\Gamma) $ a cusp form with Nebentypus $\chi$, we view the period polynomial $\rho_f$ as an
element of $W_w^{\Gamma}$. It clearly
satisfies $\rho_f(BC)=\chi(B)\rho_f(C)$ where $B$ runs through a fixed system of coset representatives for
$\Gamma\backslash\Gamma'$
and $C$ runs through a fixed system of coset representatives for
$\Gamma'\backslash\Gamma_1$. Here $\chi(B)=\chi(d_B)$ where $d_B$ is
the lower right entry of $B$ ($d_B$ is well defined modulo $N$). Let therefore $W_w^{\Gamma,\chi}$,
$C_w^{\Gamma,\chi}$ be the subspaces of $W_w^{\Gamma}$,  $C_w^{\Gamma}$ consisting of elements $P$ with
$P(BC)=\chi(B)P(C)$ with $B,C$ as above. We have orthogonal decompositions with respect to the pairing
$\{\cdot , \cdot\}$ 
\[W_w^{\Gamma}=\bigoplus_{\chi} W_w^{\Gamma,\chi}, \quad  C_w^{\Gamma}=\bigoplus_\chi C_w^{\Gamma,\chi}
\]
where the direct sums are over all characters modulo $N$. It is easily checked that the spaces
$W_w^{\Gamma,\chi}$, $C_w^{\Gamma,\chi}$ are preserved by the action of Hecke operators $\wT_n$, and the dimension
of $C_w^{\Gamma,\chi}$ equals the dimension of the Eisenstein subspace $\EE_k(N, \chi)\subset M_k(\Gamma) $. 
An indirect proof that the dimensions are equal is provided by Prop. \ref{l7.3}, which shows that
$C_w^{\Gamma,\chi}$ is dual to the space of period polynomials of Eisenstein series
$\wE_w^{\Gamma, \ov{\chi}} $ with respect to an extension of $\{ \cdot, \cdot \}$ to the space $\wW_w^{\Gamma}$ of
extended period polynomials (see also Prop. \ref{p8.9}). 

The action of $\epsilon$ also preserves these subspaces and we have the following generalization of the
Eichler-Shimura isomorphism theorem \ref{thm2.1}.
\begin{theorem}[(Eichler-Shimura)] The two maps $\rho^{\pm}:S_k(N, \chi) \rightarrow
(W_w^{\Gamma,\chi})^{\pm}$, $f\mapsto \rho_f^\pm$, give rise  to
isomorphisms, denoted by the same symbols:
 \begin{equation*}
\rho^\pm:S_k(N,\chi) \longrightarrow (W_w^{\Gamma,\chi})^{\pm}/(C_w^{\Gamma,\chi})^{\pm}.
\end{equation*}
\end{theorem}
\begin{proof}
Injectivity follows from the refinement of Haberland's formula (Theorem \ref{thm_main}) as in the proof of
Theorem \ref{thm2.1}. Surjectivity follows by considering the maps above for all characters $\chi$, and using the
fact that the sum of the dimensions of their domains equals the sum of the dimensions of their ranges.
\end{proof}

When computing spaces of modular forms, it is much more efficient to fix coset representatives $C_1, \ldots, C_d$
for $\Gamma'\backslash \Gamma_1$, with $d=[\Gamma_1:\Gamma']$, and to view $P\in W_w^{\Gamma,\chi}$ as a vector
$(P(C_1),\ldots, P(C_d))\in V_w^d$, since the values at other cosets in $\Gamma\backslash \Gamma_1$ are determined
by those above. Note however that this representation depends on the coset representatives chosen, which are fixed
once for all. 

The action of $\gamma\in\Gamma_1$ on $P\in V_w^d$ can be described as follows. Let $C_i \gamma^{-1}=B_i
C_{\sigma(i)}$ where $B_i$ are coset representatives for $\Gamma\backslash\Gamma'$ and $\sigma=\sigma_{\gamma}$ is
a permutation of $\{1, \ldots,d\}$. Then 
\[ 
P|\gamma(C_i)= \chi(B_i)P(C_{\sigma(i)})|\gamma.
\]
Similarly one can compute the actions $P|_{\Sigma} \wT_n$ with $\Sigma_n$ the double cosets giving the action of
Hecke or Atkin-Lehner operators.  
\comment{
To define the action of the Hecke operator $\wT_n$ on $P\in V_w^d$,  we use Remark \ref{r5.1}. Let $M\in M_n$.
If $C_i M^\vee$ has lower row elements $c,d$ with $\gcd(c,d,N)=1$, then we can write
\[
C_i M^\vee=M_A^\vee \gamma_i C_{\sigma(i)}
\]
where $M_A\in M_n^{\Gamma}$,  $\sigma(i)\in\{1,\ldots , d\}$ and $\gamma_i$ are coset representative for
$\Gamma\backslash\Gamma'$. In this case we have $P|M(C_i)=\chi(\gamma_i)P(C_{\sigma(i)})|M$, while if
$\gcd(c,d,N)>1$ then $P|M(C_i)=0$. }

\subsection{Computing period polynomials numerically} \label{sec5.3}

The period polynomials of a newform $f\in S_k(\Gamma_0(N),\chi)$ can be easily computed numerically using
the results of this chapter. One finds the space of period polynomials as the kernel of the matrix of period
relations, and then its plus and minus subspaces. Assuming the Hecke eigenvalues of $f$ at primes $p\nmid N$ are
known, then one can find a rational common eigenvector of enough operators $\wT_p$ acting on $(W_w^\Gamma)^\pm$
such that the resulting eigenspace is one dimensional (often one operator $\wT_p$ for the smallest prime $p\nmid N$
suffices). The period polynomials $\rho_f^\pm$ are thus determined up to scalars. The normalizing factors
can be determined from the critical values $L(w+1,f)$ (nonzero if $k>2$) and $L(w,f)$ (nonzero if $k>2$, $k\ne 4$),
unless they vanish and one has to determine $L(n,f|A)$ for some $A\in \Gamma\backslash\Gamma_1$. Once the
polynomials $\rho_f^\pm$ are determined, Theorem \ref{thm_main} gives the Petersson product $(f,f)$, and
Proposition \ref{p5.4} gives the Hecke eigenvalue $\lambda_n$ of $f$ for every $n$.

We have implemented the procedure sketched above in MAGMA \cite{Mgm} (available upon request). To give an idea
of the running time, when $f=q + 12q^3 + 88q^7 +  \ldots\in S_6(\Gamma_0(100)) $ is
a newform with rational coefficients, the computation of $\rho_f^\pm$ takes 1 sec (in this case
$\dim(W_w^\Gamma)^+=78$, $\dim(W_w^\Gamma)^-=72$, $\dim S_k(\Gamma)=66$, and $[\Gamma_1:\Gamma]=180$).
Once the polynomial $\rho_f^+$ is computed, the Hecke eigenvalue for $p=10037$ is computed in less than 1 sec using
Prop. \ref{p5.4}. The `Coefficient' command in MAGMA computes the same eigenvalue in 140 sec, while the
`coefficients' command in SAGE takes 3 sec for the same computation (on the same machine). We used the efficient
implementation by W. Stein of the coset space $ P^1(\Z/N\Z)\simeq \Gamma_0(N)\backslash\Gamma_1$ via `P1Classes',
and the Hecke elements given by `HeilbronnCremona' in MAGMA (which give the action of Hecke operators on modular
symbols, but see Remark \ref{r5.4}). 

\noindent{\bf An example.} Let $\Gamma=\Gamma_0(5)$, $k=4$, and $f=q-4q^2+2q^3+8q^4-5q^5\ldots \in S_k(\Gamma)$ the
unique normalized cusp form. In the following table, the first row contains the elements of $P^1(\Z/5\Z)$, namely
the bottom rows of a system of representatives for $\Gamma \backslash \Gamma_1$, while the second and third rows
list the components of rational generators $P_f^\pm$
of the Hecke eigenspace of $(W_w^\Gamma)^\pm$ that corresponds to the Hecke eigenform $f$ (here $P_f^-$ generates
$(W_w^\Gamma)^-$, and only the eigenvalue of $T_2$ is needed for the computation of $P_f^+$). 
  
\begin{table}[ht]\footnotesize\tabcolsep=0.21cm
\begin{tabular}{c c c c c c c}
\  & (0 1) & (1 1) & (1 3) & (1 2) &(1 4) & (1 0) \\ \hline \noalign{\smallskip}
$P_f^+$ & $-5X^2 + 1$ & $-5X^2 + 5$ & $8X^2 + 13X - 8$ & $8X^2 - 13X - 8$ & $-5X^2 + 5$ & $-X^2 + 5$\\
$P_f^-$ & $X$      & $X^2 + 2X + 1$ & $2X^2 - 3X - 2$ & $-2X^2 - 3X + 2$ & $-X^2 + 2X - 1$ & $ X $
\end{tabular} 
\end{table}

\noindent We have $\rho_f^\pm=\om_f^\pm P_f^\pm$ with
\[\om_f^+=r_{I,w}(f)=  -0.0051365773 i, \quad
\om_f^-=-wr_{I,w-1}(f)=   0.0208651386, \]
computed via \eqref{5.7} using the numerical computation of $L$-series implemented by Tim Dokchitser in MAGMA. 
Theorem \ref{thm_main} then gives $(f,f)=0.00014513335$, in agreement with the value obtained analytically in
\cite{Co}. 

Notice that $r_{A,n}^+(f)/r_{I,w}(f) \equiv 0 \pmod{13}$ for all $n$ with $0<n<w$, and all $A\in
\Gamma\backslash\Gamma_1$ (see \eqref{5.1} for notation). 
As in Manin's proof of Ramanujan's congruence mod 691 \cite{M73}, this suggests that there should be a congruence
between $f$ and an Eisenstein series. Indeed by looking at the first few coefficients\footnote{By Sturm's result
checking the first three coefficients of both sides is enough to prove the congruence.} we obtain that
$$f \equiv E_4(z)-E_4(5z)\pmod{13}$$ 
where $E_4$ is the weight 4 Eisenstein series of full level, normalized to have the coefficient of $q$ equal to 1.
A generalization of Ramanujan's congruence to Hecke eigenforms and Eisenstein series for $\Gamma_0(N)$ is in
progress.

\section{Rational decomposition of modular forms}\label{sec6}

For $\Gamma$ a finite index subgroup of $\Gamma_1$ normalized by $\epsilon$, we
give an explicit decomposition of $f\in S_k(\Gamma)$ in terms of explicit generators,
generalizing the result in the full level case \cite{Po}. For $\Gamma_1$, these are
the generators with rational periods studied in \cite{KZ}, where their periods were
first computed. For $\Gamma_0(N)$, the periods of these generators were computed for $N$
square free in \cite{An}, and for arbitrary $N$ in \cite{FY}. These generators
have explicit formulas as Poincar\'e series when $k>2$. 

To define these generators, for $A\in \Gamma\backslash\Gamma_1$ , $0\le
n\le w$, define the periods $r_{A,n}(f)$ by 
\begin{equation}\label{5.1}
\rho_f(A)(X)=\sum_{n=0}^w (-1)^{w-n}\binom{w}{n}r_{A,n}(f) X^{w-n}
\end{equation}
and similarly define $r_{A,n}^{\pm}(f)$ with $\rho_f$ replaced by $\rho_f^{\pm}$. Let
$R_{A,n}\in S_k(\Gamma)$ be the dual of the linear functional
$f\rightarrow \frac{1}{[\og_1:\og]}r_{A,n}(f)$, with respect to the Petersson product:
\[ (f, R_{A,n})=\frac{r_{A,n}(f)}{[\og_1:\og]}, \text{ for all } f\in S_k(\Gamma), 
\] 
and similarly define $R_{A,n}^{+}$, $R_{A,n}^-$. For $\Gamma=\Gamma_0(N)$, the
polynomials $\rho^-(R_{A,n}^{+})$, $\rho^+({R_{A,n}^-})$ have rational coefficients.

For $\kappa\in \{+, -\}$ and $0\le j\le w$, define the
linear combinations of periods:
\[
s_{A,j}^\kappa(f)=\sum_{n=0}^j \binom{j}{n} (-1)^{j-n} r_{A,j}^{\kappa}(f).
\] 
We then have the following generalization of Theorem 1.1 in \cite{Po}, which gives explicit
inverses of the Eichler-Shimura maps \eqref{7.1}.   
\begin{theorem}\label{thm6.1} Let $\Gamma$ be a finite index subgroup of $\Gamma_1$ normalized
by $\epsilon$, and let $\kappa_1, \kappa_2\in\{+, -\}$ with $\kappa_1\ne \kappa_2$ if $k$ even
and $\kappa_1=\kappa_2$ if $k$ odd. For $f\in S_k(\Gamma)$ 
\begin{equation*}
\frac{-3 C_{k}}{2}\cdot f=\sum_{A\in \og\backslash\og_1}
\sum_{n=0}^w
\binom{w}{n}s_{AU^{-1},n}^{\kappa_1}(f) R_{A,n}^{\kappa_2} .
\end{equation*}
\end{theorem}
\begin{proof}
Let $P\in (W_w^\Gamma)^{\kappa_1}$, $Q\in (W_w^\Gamma)^{\kappa_2}$. Then
\[\{P,\overline{Q}\}=\lla P|US-SU^2,\overline{Q}\rra 
=\lla P|U^2-U,\overline{Q} \rra 
=-2 \lla P|U, \overline{Q} \rra
\]
where the second equality follows since $P|S=-P, Q|S=-Q$, while the third from
$P|U^2-U=P|-1-2U$, together with \eqref{conj1}. 

If $R(X)=\sum_{n=0}^w (-1)^{w-n}\binom{w}{n}r_n X^{w-n}\in V_w$ then 
$$
R|U(X)=\sum_{j=0}^w\binom{w}{j} s_j X^j, \text{ with } s_j=\sum_{n=0}^j (-1)^{j-n}\binom{j}{n}
r_n.
$$ 
By Theorem \ref{thm_main} and the preceding computations it follows that for
$f,g\in S_{k}(\Gamma)$
\[
 3  C_{k}(f,g)=\{\rho_f^{\kappa_1} ,
\overline{\rho_g^{\kappa_2}}\}=-\frac{2}{[\og_1:\og]} \sum_{A\in \og\backslash
\og_1}\sum_{j=0}^w
\binom{w}{j}
s_{AU^{-1},j}^{\kappa_1}(f) \overline{r_{A,j}^{\kappa_2}(g)}.
\] 
Since $\overline{r_{A,j}^+(g)}=[\og_1:\og]\cdot (R_{A,j}^+, g)$, the claim follows.  
\end{proof}

\section{Extra relations satisfied by period polynomials of cusp forms}\label{sec7}

In this section, we determine the image of the maps 
$\rho^\pm:S_k(\Gamma)\rightarrow (W_w^\Gamma)^\pm$, namely the extra relations satisfied by
$\rho_f^\pm$ for $f\in S_k(\Gamma)$ which are independent of the period relations.  To be
explicit, the extra relations we
obtain require the determination of the periods $r_{B,m}^\mp(R_{A,w}^\pm)$ of the generators
defined in the previous section.  For $\Gamma_1$, these periods were computed in \cite{KZ},
and for $\Gamma_0(N)$, they were computed in \cite{An} (for $N$ square free, and not quite in
closed form), and in \cite{FY} (only the principal periods $r_{I,m}^\mp ({R_{I,n}^\pm})$).  
For $\Gamma_0(N)$ with small $N$, the computations in \cite{FY} are sufficient to make
completely explicit the extra relations, and we illustrate this for the case
$\Gamma=\Gamma_0(2)$. The relations are similar to the relation found by
Kohnen and Zagier in the full-level case \cite{KZ} (see also \cite[Sec. 2]{Po}). 

We first define bases of $(C_w^\Gamma)^\pm$, using the terminology in Definition
\ref{d_cusp}. For each cusp $\cc\in \Gamma\backslash \Gamma_1/\Gamma_{1\infty}$,
which is regular
if $k$ is odd, define $P_{\cc}\in C_w^\Gamma$ as
in \eqref{cw} by
fixing $A_\cc\in\cc$ a representative, and setting $c_A=(-1)^w c_{AJ}=1$ if
$[A]^+=[A_\cc]^+$, and   $c_A=0$ if $[A]\ne [A_\cc]$. Then $\{P_{\cc}\}$ form 
a basis of $C_w^\Gamma$ if $k>2$, and if $k=2$ there is only one relation $\sum_\cc P_\cc=0$.  

Assume that $\Gamma$ is normalized by $\epsilon$, and denote $A'=\epsilon
A\epsilon$. Note that if $\cc=[A]$ is a cusp, then $\cc'=[A']$ is well-defined.
Since $P_{\cc}|\epsilon=P_{\cc'}$, we have $ P_{\cc}^+=P_{\cc'}^+$ and
$P_{\cc}^-=-P_{\cc'}^-$. Therefore a basis of
$(C_w^\Gamma)^-$ consists of $P_{\cc}^-$ for each unordered pair $( \cc,\cc')$ of cusps
with $\cc\ne \cc'$, and  a basis of $(C_w^\Gamma)^+$ consists of $P_{\cc}^+$ for each 
unordered pair $( \cc,\cc')$ of cusps (when $k$ is odd only pairs of regular cusps are
considered; note that $\cc$ is regular iff $\cc'$ is regular).  

We now fix $A\in \cc$ for each (regular if $k$ odd) cusp
$\cc\in\Gamma\backslash
\Gamma_1/\Gamma_{1\infty}$ and we write, by a slight abuse of notation,
$R_{\cc,n}=R_{A,n}$, $r_{\cc,n}=r_{A,n}$ with the notation of the previous section
($R_{\cc,n}$ does depend on the choice of representative $A$, but we fix such a choice).
For each unordered pair of cusps $(\cc,\cc')$ we take  $g=R_{\cc,w}^+$ in Theorem
\ref{thm_main}, and if
$\cc\ne \cc'$ we take also  $g=R_{\cc,w}^-$ (again only for regular cusps if $k$ is odd),
obtaining the following linear relations satisfied by all $f\in S_k(\Gamma)$ \footnote{In this
section we occasionally write $\rho(f)$ instead of $\rho_f$ to simplify notation.}
\begin{equation}\label{e_rel}
\frac{3 C_k}{ [\og_1:\og] } r_{\cc,w}^+(f)= \{ \rho^+(f), \overline{\rho^-({R_{\cc,w}^+})}\},
\quad 
\frac{3 C_k}{ [\og_1:\og] } r_{\cc,w}^-(f)= \{ \rho^-(f),
\overline{\rho^+({R_{\cc,w}^-})}\}.
\end{equation}
The linear forms appearing in these relations can be applied to all $P\in
(W_w^\Gamma)^\pm$,
and putting together the relations involving $\rho^+(f)$ into a map $\lambda_+$, and the
relations involving $\rho^-(f)$ (for pairs with $\cc\ne\cc'$) into a map $\lambda_{-}$,  we
obtain two linear maps $\lambda_\pm:(W_w^\Gamma)^\pm\rightarrow
\C^{d_\pm}$ with $d_\pm=\dim(C_w^\Gamma)^\pm$. If $d_{-}=0$ the map $\lambda_{-}$ is trivial. 
 
\begin{prop}\label{p6.2} Assume $k \ge 3$ and let $\Gamma$ be a finite index subgroup of
$\Gamma_1$ normalized by $\epsilon$. With $\lambda_\pm$ defined above, we have exact sequences:
\[
0\rightarrow S_k(\Gamma)\overset{\rho^{\pm}}{\longrightarrow}
(W_w^\Gamma)^\pm\overset{\lambda_\pm}{\longrightarrow} \C^{d_\pm}\rightarrow 0.
\]
\end{prop}
\begin{proof} We have $\im \rho^{\pm}\subset \ker \lambda_\pm$ by construction. 
Note that the first relation in \eqref{e_rel} is not satisfied by $P_{\cc}^+$, while
the second is not satisfied by $P_{\cc}^-$ if $\cc\ne \cc'$, since the LHS is nonzero,
while the RHS vanishes by Lemma \ref{l4.4} a). Since $\{P_{\cc}^\pm \}$ form a basis of
$(C_w^\Gamma)^\pm$, the conclusion follows.                                
\end{proof}
\begin{remark} For $\Gamma=\Gamma_0(N)$, Proposition \ref{prop7.2} characterizes those $N$ for
which the extra relations involve only the even parts of the period polynomials. 
\end{remark}
For example, if $\Gamma=\Gamma_0(p)$ with $p$ prime, there are only two cusps $[I]$ and $[S]$,
and for all $P\in W_w^\Gamma$ we have $P(S)=-P(I)|S$ by the period relations. In
particular $r_{S,w}(f)=-r_{I,0}(f)$. Noting also that $P(I)^+$, $P(S)^+$ are the even parts of
$P(I), P(S)$, so that $R_{I,n}^+= R_{I,n} $ for $n$ even, we have the
following simpler version
of Proposition \ref{p6.2}. 

\begin{corollary}\label{c6.4} Let $\Gamma=\Gamma_0(p)$ with $p$ prime, and let $k>2$ even. Then
the two extra relations satisfied by all even period polynomials $\rho^+(f)$ for $f\in
S_k(\Gamma)$ are
\[\frac{3 C_k}{ [\og_1:\og] } r_{I,a}(f)= \{ \rho^+(f),
\overline{\rho^-({R_{I,a}})}\},\quad
\text{ for } a=0,w. 
\] 
\end{corollary}

For small values of $N$ (eg $N=2,3,4,5$), the polynomials $\rho_f$ are completely determined
by
the principal parts $\rho_f(I)$, so that the relations above are completely explicit via the
computation of $\rho^-(R_{I,a})(I)$ in \cite{FY}. In the remainder of this section, we discuss
in detail the case of period polynomials for $\Gamma=\Gamma_0(2)$, which have been studied in
\cite{IK},\cite{FY},  \cite{KT}. 

We take as coset representatives for $\Gamma\backslash \Gamma_1$ the set
$\{I,U,U^2\}$. Denoting by $\overline{A}$ the coset
$\Gamma A$, we have $\ov{S}=\ov{U}$, $\ov{US}=\ov{I}$, $\ov{U^2 S}=\ov{U^2}$. For
$P\in W_w^\Gamma$, the
period relations $P|1+S=0$, $P|1+U+U^2=0$ reduce to
\begin{equation}\label{7.2}
P(U)+P(I)|S=0,\quad P(U^2)|1+S=0, \quad P(U^2)+ P(U)|U+P(I)|U^2=0. 
\end{equation}
The polynomials $P(U), P(U^2)$ are therefore determined by $P(I)$ which satisfies the relation
\begin{equation}\label{per2}
 P(I)|(ST-ST^{-1})(1+S)=0.
\end{equation}
Let $U_w\subset V_w$ denote the set of polynomials satisfying \eqref{per2}, so that we can
identify $W_w^\Gamma$ with $U_w$ via $P\rightarrow P(I)$. Conjugation by
$\epsilon$ leaves unchanged each coset $\ov{I}, \ov{U}, \ov{U^2}$, hence $P^+, P^-$ 
correspond to the even and odd parts of the polynomial $P(I)$ in this identification.

 To express the formula in Theorem \ref{thm_main} in terms of  $\rho_f(I), \rho_g(I)$ alone,
let $P\in (W_w^\Gamma)^+, Q\in
(W_w^\Gamma)^-$. We have $\lla P|T-T^{-1}, Q \rra=-2\lla P|U, Q \rra$ as in the proof of
Theorem
\ref{thm6.1}, and using \eqref{7.2} we obtain
\[
 \lla P|U, Q \rra=\frac{1}{3}<P(I)|2T^{-1}-2I-T, Q(I)>=-\frac{1}{2} <P(I)|T-T^{-1}, Q(I)> 
\]
where $\la P(I),Q(I)\ra=\la P(I)|T+T^{-1}, Q(I)\ra =0$ since $P(I),Q(I)$ have opposite
parity. We can therefore restate Theorem \ref{thm_main} as follows, setting $P_f=\rho_f(I)\in
U_w$
\begin{equation}\label{7.5}
 3 C_{k}\cdot (f,g)=\la P_f^+|T-T^{-1}, \ov{P_g^-}\ra. 
\end{equation}

By Proposition \ref{prop7.2} and the Eichler-Shimura isomorphism \eqref{7.1}, the map $f
\rightarrow P_f^{-}$ gives an isomorphism 
$S_{k}(\Gamma_0(2))\simeq U_w^-$, while the image of the map  $f\rightarrow P_f^+$ is a
codimension 2 subspace of $U_w^+$. \footnote{A direct proof in this case is contained in
\cite[Theorem
4]{KT}.} To simplify notation, let $r_n(f)=r_{I,n}(f)$, and $R_n=R_{I,n}$ for $0\le n
\le w$. 
\begin{corollary}\label{c6.5}
 For $f$ in $S_k(\Gamma_0(2))$, let 
$
s_n(f)=\displaystyle\sum_{\substack{j=0\\ n-j \text{ odd}}}^n\binom{n}{j} r_j(f)
$. The extra relations satisfied by the even periods of $f$ are
\[
r_a(f)=\sum_{\substack{n=0\\ n \text{ odd}}}^w\binom{w}{n} s_{w-n}(f)
\frac{2}{C_{k}} r_{a}(R_{n})\quad \text{ for  } a=0,w.
\]
\end{corollary}
\noindent From \cite{FY}, for $0<w<n$, $n$ odd, we have
$r_{w}(R_{n})=-\displaystyle\frac{r_0(R_{\nt})}{N^n}$ and 
\[
\frac{2}{C_{k}}r_0(R_{n})=-N^{\nt}\frac{B_{\nt+1}}{\nt+1}+\frac{k}{B_k}\frac{B_{ n+1 } }
{ n
+1} \frac{B_{\nt+1}}{\nt+1}\frac{\alpha_{N,k}(n)}{N}+\frac{\delta_{w,n+1}}{w},
\]
where $\nt=w-n$, $\alpha_{N,k}(n)=\frac{1-N^{-n-1}}{1-N^{-k}}$ (recall $N=2$), and $B_m$ are
the Bernoulli numbers. Note that there is a
minus sign missing in the normalization of the generators denoted by $R_{\Gamma,w,n}$ in
\cite[Def. 1.1]{FY}, and with this correction we have $R_n=-\frac{C_{k}}{2}
R_{\Gamma,w,n}$. 
\begin{proof} 
Since $P_f|T-T^{-1}(X) =-2 \sum_{n=0}^w (-1)^n \binom{w}{n}s_{n}(f) X^{w-n},$
and $\ov{r_{n}(R_{0})}=r_0(R_{n})$, the claim follows from Corollary \ref{c6.4}. 

\end{proof}

The periods $r_n(f)$ are related to the critical values of the $L$-series
associated to $f$, and when $f$ is a newform they can be readily computed using MAGMA \cite{Mgm}. The relations
in Corollary \ref{c6.5} have been checked numerically for $k=8,10,14$.

\section{Period polynomials of arbitrary modular forms}\label{sec8}

In this section we define period polynomials for noncuspidal modular forms, and extend
Haberland's formula and the action of Hecke operators to the larger space  of
period polynomials of all modular forms. An important feature of the
larger space $\wW_w^\Gamma$ is that the 
the pairing $\{ \cdot,\cdot \}$ has a natural nondegenerate extension to it, while
on $W_w^\Gamma$ it is degenerate (its radical is $C_w^\Gamma$). If $\Gamma=\Gamma_1(N)$ and
$k>2$, the period polynomial maps $\rho^{\pm}$ extend naturally to the larger space, and
they give isomorphisms between $M_k(\Gamma)$ and $(\wW_w^\Gamma)^\pm$. Surprisingly, when
$k=2$, $\Gamma=\Gamma_0(N)$ and $N$ is squarefree with at least two prime factors, only one of
the two maps is an isomorphism.

For the full modular
group, period polynomials of Eisenstein series were defined in \cite{KZ}, using the description
of periods as special values of the associated $L$-function, and the enlarged space of period
polynomials was introduced in \cite{Z91}. A different construction using an Eichler integral
was given more recently in \cite{BGKO}, in the more general context of weakly holomorphic
modular forms. We extend both the Eichler integral and the $L$-function approach to a finite
index subgroup $\Gamma$ of $\Gamma_1$.

For $f\in M_k(\Gamma)$, we define $\wr_f=\tf|1-S$ as in \eqref{e_int} (with the same action as
on period polynomials), where the Eichler
integral $\tf:\Gamma \backslash \Gamma_1 \rightarrow \A$ is given by
\begin{equation}\label{a1}
 \tf(A)(z)=\int_{z}^{i\infty}[f|A(t)-a_0(f|A)] (t-z)^w dt  
\end{equation}
with $a_0(f|A)$ the constant term of the Fourier expansion of $f|A$. Note that
$a_0(f|A)=a_0(f|AT)$, so $\tf|1-T=0$. Let
 $$\wV_w:=\Big\{\sum_{-1\le i \le w+1} a_i X^{i}: a_i\in \C \Big\},$$  
and let $\wV_w^\Gamma$ be the space of functions $P:\Gamma \backslash \Gamma_1 \rightarrow
\wV_w$ with $P(-A)=(-1)^wP(A)$ for $A \in \Gamma \backslash \Gamma_1$.
We define an action of $g\in \Gamma_1$ on $P\in \wV_w^\Gamma$ by $P|g(A)=P(Ag^{-1})|_{-w} g$
as before. Note that this is no longer well defined in general, as elements of
$\Gamma_1$ do not preserve the space $\wV_w$. However one can still define the subspace
$$
\wW_w^\Gamma=\{P\in \wV_w^\Gamma: P|1+S=P|1+U+U^2=0 \}.
$$
We will show below that $\wr_f \in \wW_w^\Gamma$. That it satisfies the period relations is
immediate: we have $\wr_f|1+S=0$, and $\wr_f|1-T=0$ from the definition, while 
\[
\wr_f|1+U+U^2 = \tf | (1-S)(1+U+U^2)= \tf|(1-T^{-1})(1+U+U^2)=0. 
\]
It remains to show that $\wr_f\in \wV_w^\Gamma$, and we will do this by relating it with the
polynomial $\rho_f \in V_w^\Gamma$ defined by \eqref{5.1}, where 
\begin{equation}\label{e7.2}
r_{A,n}(f)=(-1)^{n+1} \frac{\Gamma(n+1)}{(2\pi i)^{n+1}} L(n+1,f|A).
\end{equation}
The $L$-function $L(s,f)=\sum_{n=1}^\infty a_n(f) n^{-s}$
is given, if $\re(s)>k$, by the Mellin transform 
$$
(-1)^s (2\pi i)^{-s}\Gamma(s) L(s,f)=\int_0^{i\infty} [f(t)-a_0(f)]t^{s-1} dt,$$
and it can be extended meromorphically to $\C$ by fixing $z_0\in\H$, decomposing
$\int_0^{i\infty}=\int_0^{z_0}+\int_{z_0}^{i\infty}$, and making a change of
variables $t=S u $ in the first integral. We obtain a meromorphic function with at most
simple
poles at $s=0$ and $s=k$:
\begin{equation*}
\begin{split}
\frac{(-1)^{s}\Gamma(s)}{(2\pi i)^s}L(s,f)=\int_{z_0}^{i\infty} [f(t)-a_0(f)]t^{s-1}
dt+(-1)^s\int_{S z_0}^{i\infty} [f|S(t)-a_0(f|S)]t^{k-s-1} dt-\\
-a_0(f)\frac{z_0^s}{s}-(-1)^s a_0(f|S) \frac{(Sz_0)^{k-s}}{k-s}.
\end{split} 
\end{equation*}
Introducing as in \cite{KZ} the function $H_{z_0}\in V_w^\Gamma$ defined for $A\in
\Gamma\backslash \Gamma_1$ by
\begin{equation}\label{a10}
H_{z_0}(A)=\int_{z_0}^{i\infty} [f|A(t)-a_0(f|A)](t-X)^w
dt-a_0(f|A)\int_0^{z_0} (t-X)^w dt\in V_w
\end{equation}
we obtain from \eqref{5.1} and the analytic continuation above
\begin{equation}\label{a11}
 \rho_f(A)=H_{z_0}(A)-H_{S z_0}(AS^{-1})|S,
\end{equation}
namely $\rho_f=H_{z_0}-H_{S z_0}|S$. 

We now determine the relation between $\wr_f$ and $\rho_f$, which also shows that $\wr_f\in
\wW_w^\Gamma$.

\begin{prop} \label{pa1} For $f\in M_k(\Gamma)$, let $\rho_f^0 \in \wV_w^\Gamma$ be given by
$\rho_f^0(A)=(-1)^w\frac{a_0(f|A)}{w+1} X^{w+1}$. We have
\[\wr_f=\rho_f+ \rho_f^0|(1-S) , 
\]
namely $\wr_f(A)=\rho_f(A)+(-1)^w\frac{a_0(f|A)}{w+1} X^{w+1}+\frac{a_0(f|AS^{-1})}{w+1}
X^{-1}$. 
\end{prop}
\begin{proof} Fixing $z_0\in \H$ and decomposing the integral in \eqref{a1} as
$\int_{z_0}^{i\infty}+\int_{z}^{z_0}$ we have
\[\tf(A)(z)=H_{z_0}(A)(z)+\int_{z}^{z_0} f|A(t) (t-z)^w dt +a_0(f|A)\int_0^{z}(t-z)^w dt.
\]
Using the same relation for $\tf(AS^{-1})(Sz)j(S, z)^w$ with $Sz_0$ in
place of $z_0$ we obtain, after a change of variables $u=St$ in the first integral above
\begin{equation*}
\begin{aligned}
 (\tf|1-S) (A)(z)=& H_{z_0}(A)(z)-H_{Sz_0}(AS^{-1})|S (z)+\\
&\quad +a_0(f|A)\int_0^{z} (t-z)^w dt -a_0(f|AS^{-1})\int_0^{Sz} (t-Sz)^w j(S,z)^w dt.
\end{aligned}
\end{equation*}
Computing the last integrals and comparing with \eqref{a11} yields the conclusion. 
\end{proof}

We now determine the exact relationship between $\wW_w^\Gamma$ and $W_w^\Gamma$. For $\wP\in
\wW_w^\Gamma$ write $\wP=P+P_0$ where $P\in V_w^\Gamma$ and $P_0(A)=c_A X^{w+1}+d_A X^{-1}$
for $A\in \Gamma\backslash\Gamma_1$.
From $\wP|1+S=0$ we obtain $d_A=c_{AS}$. From $\wP|1+U+U^2=0$, it follows that $P_0|1+U+U^2\in
V_w^\Gamma$, which implies that $c_A=c_{AT}$ for all $A\in \Gamma_1$. Therefore we have
$P_0=P^0|1-S$, where $P^0(A)=c_A X^{w+1}$, with $c_A=c_{AT}$. In conclusion, letting 
\begin{equation}\label{a2}
D_w^\Gamma=\{(c_A X ^{w+1})_A | (1-S) :   c_A=c_{AT}=(-1)^w c_{AJ}\in \C \}\subset
\wV_w^\Gamma
\end{equation}
we have a unique decomposition of $\wP\in\wW_w^\Gamma$ as above
\begin{equation}\label{a3}
\wP=P+P^0|(1-S) , \quad P\in V_w^\Gamma, \ P^0|(1-S)\in D_w^\Gamma .
\end{equation} 
For $\wr_f$ this is the decomposition in Proposition \ref{pa1},
since $a_0(f|A)=a_0(f|AT)=(-1)^w a_0(f|AJ)$. As in the proof of Lemma \ref{L7.1}, note that
$\dim D_w^\Gamma$ equals $e_\infty(\Gamma)$ or $e_\infty^\reg(\Gamma)$, depending on whether
$k$ is even or odd respectively. 

When $k=2$ there is an extra relation satisfied by the coefficients of $P^0$ in \eqref{a3}.
Letting $P(A)=d_A\in\C$, $P^0(A)=c_A X^{w+1}$, the period relations now imply that 
\[
d_A+d_{AU}+d_{AU^2}+2(c_A+c_{AU}+c_{AU^2})=0, \text{ with } d_A+d_{AS}=0, c_A=c_{AT}
 \]  
for all $A\in \Gamma\backslash \Gamma_1$. The relation $d_A+d_{AS}=0$ implies $\sum_A d_A=0$, and then the first
relation above implies that $\sum_{A} c_A=0$ as well, where the sum is over a complete system of representatives
for $\Gamma\backslash\Gamma_1$. From Proposition \ref{pa1} it follows that $\sum_A a_0(f|A)=0$ for all $f\in
M_2(\Gamma)$.

\begin{prop}\label{pa2} a) If $k \ge 3$ there is an exact sequence
\[
0 \rightarrow W_w^\Gamma \rightarrow \wW_w^\Gamma \rightarrow D_w^\Gamma\rightarrow 0
\]
where the first map is inclusion, and the second is the map $\wP\rightarrow P^0|1-S$ defined
above.

b) If $k=2$ there is an exact sequence
\[
0 \rightarrow W_w^\Gamma \rightarrow \wW_w^\Gamma \rightarrow D_w^\Gamma\rightarrow
\C\rightarrow 0
\]
where the last map takes $P^0|1-S\in D_w^\Gamma$ with $P^0(A)=c_A X^{w+1}$ to $\sum_{A\in
\Gamma\backslash \Gamma_1} c_A$.
\end{prop}
\begin{proof} Exactness at $\wW_w^\Gamma$ follows from the definition. If $k\ge 3$,
surjectivity of the second map 
follows from Proposition \ref{a1}, and the fact that there is a basis of Eisenstein series
$E_k^{\cc}\in M_k(\Gamma)$ for $\cc=[A_\cc]$ a complete system of representatives for the 
(regular if $k$ is odd) cusps in $\Gamma\backslash \Gamma_1/\Gamma_{1\infty}$, such that
$a_0(E_k^{\cc}|A)=(-1)^w a_0(E_k^{\cc}|AJ)$ equals 1 if $[A]^+=[A_\cc]^+$, and 0 if $[A]\ne\cc$
 (see Definition \ref{d_cusp} for notation). If $k=2$ the Eisenstein subspace of
$M_2(\Gamma)$ is spanned by
modular forms
$f$ which are nonzero only at a fixed pair of nonequivalent cusps and are zero at other cusps,
and such that $\sum_{A\in \Gamma\backslash\Gamma_1} a_0(f|A)=0$.  Their images in $D_w^\Gamma$
span the kernel of the last map, proving exactness at $D_w^\Gamma$.
\end{proof}
The previous proposition shows that $\dim \wW_w^\Gamma=2 \dim M_k(\Gamma)$. From the
proof we conclude that there is a direct sum decomposition 
\begin{equation}\label{7.6}
\wW_w^\Gamma=W_w^\Gamma\oplus\wE_w^\Gamma
\end{equation}
where $\wE_w^\Gamma$ is the image of the Eisenstein subspace
$\EE_k(\Gamma)\subset M_k(\Gamma)$ under the map $f\rightarrow \wr_f$. 

The pairing $\{\cdot, \cdot \}$ extends to a pairing on $\wW_w^\Gamma \times\wW_w^\Gamma$, by 
decomposing $\wP,\wQ\in \wW_w^\Gamma$ as in \eqref{a3} and setting:
\begin{equation}\label{a7}
 \{\wP,\wQ \}=\lla P|T-T^{-1}, Q \rra +\lla 2P^0|T-T^{-1},Q \rra+  \lla P, 2 Q^0|T^{-1}-T
\rra+I_k(P^0, Q^0),
\end{equation}
where $I_k(P^0, Q^0)=0$ if $k$ even and $I_k(P^0, Q^0)=\frac{6 (k-1)
}{k[\Gamma_1:\Gamma]}\sum_{A\in \Gamma\backslash \Gamma_1}{c_A c_A'}$ if $k$ is odd, where $P^0(A)=c_A
X^{w+1}$, $Q^0(A)=c_A' X^{w+1}$.  

Since  $P^0|T-T^{-1}\in V_w^\Gamma$ this pairing is well-defined, and it is easily checked
that it behaves as in \eqref{conj1} under the action of $\epsilon$ defined as in \eqref{eps}.
We will show below that this definition is natural for two reasons:
Haberland's formula generalizes to arbitrary modular forms, if the Petersson product is
extended in a natural way to all modular forms, and this pairing is Hecke equivariant
for $\Gamma=\Gamma_1(N)$, with the same action of Hecke operators on $\wW_w^\Gamma$ as on
$W_w^\Gamma$. 

Recall that on $W_w^\Gamma$ the pairing $\{\cdot, \cdot \}$ is degenerate, its radical being
$C_w^\Gamma$ (Lemma \ref{l4.4}). We now show that the extended pairing is nondegenerate on
$\wW_w^\Gamma$, more precisely that the dual of $C_w^\Gamma$ inside $\wW_w^\Gamma$ is the space
$\wE_w^\Gamma$. 

\begin{prop}\label{l7.3}
 $\mathrm{(a)}$  Let $P=P'|1-S \in C_w^\Gamma$ and $\wQ=Q+Q^0|1-S \in \wW_w^\Gamma$, and let 
$P'(A)=c_A'$, $Q^0 (A)=(-1)^w c_A \frac{X^{w+1}}{w+1}$ for $A\in
\Gamma\backslash\Gamma_1$ (so that $c_A'=c_{AT}'=(-1)^w c_{AJ}',
c_A=c_{AT}=(-1)^w c_{AJ}$). Then 
\[
\{ P, \wQ\}=-\frac{6}{[\og_1:\og]} \sum_{A\in \og\backslash \og_1} c_A' c_A  .
\]
$\mathrm{(b)}$   The pairing $\{ \cdot, \cdot\}$ is nondegenerate on $\wW_w^\Gamma$, and the dual of $C_w^\Gamma$
is $\wE_w^\Gamma$. 
\end{prop}
\begin{proof} 
(a) As in the proof of Lemma \ref{l4.4}, we use the formal relation \eqref{4.8}, together
with the relation $(1-S)(1+U+U^2)=(1-T^{-1})(1+U+U^2)$ and the $\Gamma_1$ invariance of the pairing $\lla \cdot
,\cdot \rra$:
\begin{equation*}
\begin{aligned}
\{P,\wQ\}&=\lla P'|(1-S)(T-T^{-1}), Q \rra+2\lla P', Q^0|(T^{-1}-T)(1-S) \rra \\
 &= 2\lla P', Q|1+U+U^2  \rra +  2\lla P', Q^0|(T^{-1}-T)(1-S) \rra \\
 &=-2\lla P', Q^0|[(1-T^{-1})(1+U+U^2)+(T-T^{-1})(1-S)] \rra \\
 &=-2\lla P', Q^0|[2(1-T^{-1})+T-1 ]\rra \\
 &=-\frac{6}{[\og_1:\og]} \sum_A c_A' c_A. 
\end{aligned}
\end{equation*}

(b) We choose a basis of $\wW_w^\Gamma$ by concatenating bases for $C_w^\Gamma,
\rho^-(S_k(\Gamma))+ \rho^+(S_k(\Gamma)), \wE_w^{\Gamma}$ in this order. 
The block form matrix of the pairing $\{\cdot , \cdot\}$ with respect to this basis is
\[\left(\begin{smallmatrix} 0 & 0& \mathcal{A} \\0 & B & 0 \\(-1)^{w+1}\mathcal{A}^t & 0& C
    \end{smallmatrix}\right),
\]
so it is enough to show that $\A$ is nonsingular ($B$ is nonsingular by Theorem
\ref{thm_main}). When $k>2$ this is obvious from part (a). For $k=2$, we fix a cusp $\cc_0$ in
$\Gamma\backslash\Gamma_1/\Gamma_{1\infty}=\{\cc_0, \cc_1, \ldots, \cc_n\}$ and we let a basis
of $C_w^\Gamma$ consist of $P_i$, $1\le i \le n$, as in the statement of the lemma, with the
constants $c_{iA}'=1$ if $[A]\in \cc_i$, and $c_{iA}'=0$ otherwise (note that
$P_0=-\sum_{i=1}^nP_i$). Letting $l_i=\#\{A\in \og\backslash\og_1: [A]\in \cc_i \}$ (the width
of the cusp $\cc_i$),  we take a basis of $\wE_w^\Gamma$ to consist of $\wQ_i$ as in the
statement, $1\le i \le n$, with $c_{iA}=1$ if $[A]\in \cc_i$, $c_{iA}=-\frac{l_i}{l_0}$ if
$[A]\in \cc_0$, and $c_{iA}=0$ otherwise. With respect to this basis the matrix $\A$ is
diagonal, so the pairing is nondegenerate. 
\end{proof}
Assume now that $\Gamma$ is normalized by $\epsilon$. Since
the action of $\epsilon$  given by \eqref{eps} preserves $\wW_w^\Gamma$ and $ D_w^\Gamma$,
passing to the $\pm 1$ eigenspaces in Proposition \ref{pa2} gives exact sequences 
\begin{equation}\label{7.7}
0 \rightarrow (W_w^\Gamma)^\pm \rightarrow (\wW_w^\Gamma)^\pm \rightarrow
(D_w^\Gamma)^\pm\rightarrow \C \rightarrow 0,   
\end{equation}
where the last map is nontrivial only if $k=2$ and the sign is minus, when it is defined in
Proposition \ref{pa2} b) (when $k=2$ and $P^0|1-S \in (D_w^\Gamma)^+$ with $ P^0(A)=c_A
x^{w+1}$, then $c_A=-c_{A'}$ and $\sum_A c_A=0$ automatically).  
From \eqref{cw} and \eqref{a2} we see that $\dim (D_w^\Gamma)^+=\dim (C_w^\Gamma)^-$ for all
$k$; $\dim (D_w^\Gamma)^-=\dim (C_w^\Gamma)^+$  for $k\ge 3$; and  $\dim
(D_w^\Gamma)^- -1=\dim(C_w^\Gamma)^+$  for $k=2$. Combined with the Eichler-Shimura
isomorphism \eqref{7.1} and Lemma \ref{L7.1}, this implies that $\dim (\wW_w^\Gamma)^\pm =\dim
M_k(\Gamma)$.                      

The next Proposition can be seen as a extension of the Eichler-Shimura isomorphism \eqref{7.1}
to the entire space of modular forms.
\begin{prop}\label{p7.4} $\mathrm{(a)}$ Assume that $k, \Gamma$ are such that the extended Petersson scalar
product on $M_k(\Gamma)$ defined in $\S$\ref{s7.1} is nondegenerate (see Remark \ref{r7.5}).
Then the maps $\wr^\pm:
M_k(\Gamma) \rightarrow (\wW_w^\Gamma)^\pm$, $f\rightarrow \wr_f^\pm$, are isomorphisms. 

$\mathrm{(b)}$ Assume that $(C_w^\Gamma)^-=0$ (for example $\Gamma=\Gamma_0(N)$ with $N$ as in
Proposition \ref{prop7.2}). Then $\wr^-$ is an isomorphism. 
\end{prop}
\begin{remark}\label{r7.5}
It is shown in \cite{PP12} that the extended Petersson product is nondegenerate for
$\Gamma_1(N)$ (and therefore also for $\Gamma_0(N)$) when $k>2$. When $k=2$ the extended
Petersson product is nondegenerate for $\Gamma_1(p)$ or $\Gamma_0(p)$ 
with $p$ prime, while it is degenerate for $\Gamma_0(N)$ with $N$ squarefree with at least two
prime factors. This implies that in the latter case the map $\wr^+$ is not an isomorphism;
indeed, part b) shows that $\wr^-$ is an isomorphism, and if both $\wr^{\pm}$ were
isomorphisms, then the
Petersson product would be nondegenerate by Theorem \ref{pa4} c) below, since the pairing
$\{\cdot, \cdot\}$ is nondegenerate. 
\end{remark}
\begin{proof} a) Since the dimensions of the spaces are equal, we only have to prove
injectivity.
If $\wr_f^\pm=0$, it follows from Theorem \ref{pa4} c) that $f=0$, when
 the extended Petersson product on $M_k(\Gamma)$ is nondegenerate. 

b) If $(C_w^\Gamma)^-=0$, then $(D_w^\Gamma)^+=0$ as well. Assuming  $\wr_f^-=0$ for $f\in
M_k(\Gamma)$, it follows that $\rho_f^0=0$ so $f$ is a cusp form, hence $f=0$ by Theorem
\ref{thm_main}. 
\end{proof}

\subsection{An example} As an example, we check directly that the map
$\wr^+$ is not an isomorphism for $k=2$ and $\Gamma=\Gamma_0(6)$. As explained in Remark
\ref{r7.5}, this gives an alternative proof that the extended Petersson product is degenerate
in this case, in agreement with \cite{PP12}. 

As representatives $A_j$ for $\Gamma\backslash \Gamma_1$ we take the matrices
\[
 ST^{-i}S\{I, U^2, U\},  \quad  i=0,1,2,3 
\] 
in this order (namely $(A_1,\ldots, A_{12} )= (I, U^2, U, ST^{-1}S, ST^{-1}S U^2, \ldots,
ST^{-3}S U )$), obtained from the set of representatives provided by the command
`CosetRepresentatives' in MAGMA. There are four cusps $\cc_i\in
\Gamma \backslash \Gamma_1/\Gamma_{1\infty} $, and we have $\cc_1=[A_1]$,
$\cc_2=[A_9]=[A_{12}]$,
$\cc_3=[A_6]=[A_7]=[A_{11}]$, while the remaining six matrices are in the class $\cc_4$.   

Since there are no cusp forms of weight two on $\Gamma_0(6)$, we have $W_w^\Gamma=C_w^\Gamma$.
The latter space is spanned by polynomials $P_i$ supported at the class $\cc_i$, $i=1,2,
3$, namely $P_i=(c_A)_A|1-S$ with $c_A=c_{AT}$ and $c_A=1$ if $[A]=\cc_i$, $c_A=0$ otherwise.  
We identify a polynomial $P\in C_w^\Gamma$ with a vector ${\bf d}=(d_i)\in \C^{12}$ with
$P(A_i)=d_i$. Let $\sigma \in \SS_{12}$ be the permutation such that $A_j S=A_{\sigma j}$. We
have $\sigma= (3, 4, 1, 2, 7, 10, 5, 12, 11, 6, 9, 8), $
and it follows that the vectors ${\bf d}$ corresponding to the polynomials $P_1$, $P_2$, $P_3$
are given respectively by (the entries not specified are equal to 0):
\[
 d_1=1, d_3=-1;\quad d_9=d_{12}=1, d_{11}=d_{8}=-1; \quad d_6=d_7=d_{11}=1,
d_{10}=d_5=d_{9}=-1.
\]
Therefore in order to decompose a polynomial $P\in C_w^\Gamma$ with respect to the basis
$\{P_1, P_2, P_3\}$ it is enough to know $d_1$, $d_{12}$ and $d_9$.

The space $M_2(\Gamma_0(6))$ is spanned by the Eisenstein series $E_2^t(z)=E_2(z)-tE_2(tz)$,
for $t=2,3,6$, where $E_2(z)=-\frac{1}{24}+\sum_{n\ge 1} \sigma_1(n) e^{2\pi i nz}$.  Since
$(D_w^\Gamma)^+=(C_w^\Gamma)^-=0$, we have $\wr^+(E_2^t)=\rho^+(E_2^t)\in C_w^\Gamma$. Letting
$\rho(E_2^t)(A_i)=e_i$, $\rho(E_2^t)(A_i')=e_i'$, we have
$\rho^+(E_2^t)(A_i)=\frac{e_i+e_i'}{2}=d_i$, where $e_j'=e_{\tau j}$ with $\tau=(1,
4, 3, 2, 10, 7, 6,  8,  9, 5, 11,12)\in \SS_{12}$. 

We now determine the constants $d_i$ for
each of the three Eisenstein series. Taking into account that $L(s, E_2)=\zeta(s)\zeta(s-1)$
and $L(s, E_2^t)=\zeta(s)\zeta(s-1)\big(1-\frac{1}{t^{s-1}}\big)$, 
the constant $d_1$ is given by \eqref{e7.2}: 
$$ d_1= C \ln(t), \quad  \text{ where }C= -\frac{\zeta(0)}{2\pi i}.$$

Since $E_2^2\in M_2(\Gamma_0(2))$, and $A_6,A_7, A_{11}\in \Gamma_0(2) $ it follows that
$e_1=e_6=e_7=e_{11}$. We also have $e_6'=e_7$, $e_{11}'=e_{11}$, hence
$d_1=d_6=d_7=d_{11}$. We obtain $\rho^+(E_2^2)=C \ln(2)(P_1+P_3).$

Since $E_2^3\in M_2(\Gamma_0(3))$, and $A_9,A_{12}\in \Gamma_0(3) $, we have $d_1=d_9=d_{12}$
and $\rho^+(E_2^3)=C \ln(3)(P_1+P_2). $ 

For $E_2^6$, in order to determine $d_9, d_{12}$ we easily find 
\comment{
so writing $E_2^6=E_2^3+ 3 E_2^2 |
\big(\begin{smallmatrix} 3 & 0 \\ 0 & 1 \end{smallmatrix}\big),$  it follows that 
\[E_2^6|A_9=E_2^3+ 3 (E_2^2 |S)|
\big(\begin{smallmatrix} 3 & 0 \\ 0 & 1\end{smallmatrix}\big), \quad E_2^6|A_{12}=E_2^3+ 3
(E_2^2 |ST)|
\big(\begin{smallmatrix} 3 & 0 \\ 0 & 1\end{smallmatrix}\big).
\]
From the transformation properties of $E_2$ under $\Gamma_1$ we have
$E_2|S(z)=E_2(z)-\frac{1}{2}E_2\big(\frac{z}{2}\big)$, so
$E_2^6|A_9(z)=E_2(z)-\frac{3}{2}E_2\big(\frac{3z}{2}\big)$ and
$E_2^6|A_{12}=E_2(z)-\frac{3}{2}E_2\big(\frac{3z+1}{2}\big)$, obtaining 
}
\[L(s, E_2^6|A_9)=\zeta(s)\zeta(s-1) \big(1-\textstyle\frac{2^{s-1}}{3^{s-1}}\big), \ \ L(s,
E_2^6|A_{12})=\zeta(s)\zeta(s-1)
\big(1-3^{2-s}+ \frac{2^{s-1}}{3^{s-1}}+6^{1-s}\big) .
\]
By \eqref{e7.2} it follows $d_9=C (\ln(3)-\ln(2))$ , $d_{12}= C \ln 3$, so that
\[ 
\rho^+(E_2^6)=C( P_1\ln 6 +P_2\ln 3  + P_3\ln 2 )= \rho^+(E_2^2)+\rho^+(E_2^3) 
\]
concluding that $\wr^+$ is not surjective.

\subsection{Haberland's formula for arbitrary modular forms}\label{s7.1}

The Petersson scalar product of two Eisenstein series of full level is defined by Zagier
in \cite{Z81}. Let
$\F$ be the fundamental domain $\{z\in \H:  |z|\ge 1, |\re\, z |\le 1/2  \}$ for $\Gamma_1$,
and for $T>1$ let $\F_T$ be the truncated domain for which $\im\, z <T$. Since $\sum_A
f|A(z)\ov{g}|A(z) y^k$ is a $\Gamma_1$-invariant, renormalizable function in the sense of 
\cite{Z81}, we can define for $f,g\in M_k(\Gamma)$ 
\begin{equation}\label{petersson}
(f,g)=\frac{1}{[\og_1:\og]}\lim_{T\rightarrow \infty}\sum_A\Big[ \int_{\F_T}
f|A(z)\overline{g|A(z)} y^w  dx dy -\frac{T^{k-1}}{k-1} a_0(f|A)\overline{a_0(g|A) }\Big] 
\end{equation}
where the sum is over a complete system of representatives $A\in\og\backslash \og_1$. As in
\cite{Z81}, it can be shown that the extended Petersson product equals $\res_{s=k} L(s, f,
\overline{g})$ up to a nonzero constant, where $ L(s, f,
\overline{g})=\sum_{n\ge 1} \frac{a_n(f)\overline{a_n(g)}}{n^s} $. Using this fact, we show in
\cite{PP12} that for $\Gamma=\Gamma_1(N)$ the extended Petersson product is nondegenerate 
when $k>2$, while for $k=2$ it may or may not be degenerate.

We have the following generalization of Theorems \ref{thm_hab} and \ref{thm_main}.
\begin{theorem} \label{pa4} Assume $k\ge 2$, and $\Gamma$ is a finite index subgroup of
$\Gamma_1$. Let $f,g\in M_k(\Gamma)$.

a) We have: 
$\ \ 
6 C_k\cdot (f,g) = \{\wr_f, \overline{\wr_g}\},\ \ 
$ where $C_k=-(2i)^{k-1}$. 

b) We have: $\ \{ \wr_f, \wr_g \}=0.$

c) Assuming further that $\Gamma$ is normalized by $\epsilon$, and letting $\kappa_1,
\kappa_2\in\{+,-\}$ as in Theorem \ref{thm_main}: 
$$
3 C_k\cdot (f,g) = \{\wr_f^{\kappa_1}, \overline{\wr_g^{\kappa_2}}\}.
$$
\end{theorem}
\begin{proof} a) If one of $f,g$ is a cusp form, then we can apply Stokes' theorem over the
fundamental domain $\D$ for $\Gamma(2)$ as in the  proof of Theorem \ref{thm_hab}, and
easily obtain the desired identity. When both $f,g$ have nonzero constant terms, this approach
is complicated by the fact that both $f$ and $g$ blow up at the cusps $-1,0,1$, and we prefer
to apply Stokes' theorem to the domain $\F$ as in \cite{KZ}. We use the following
abbreviations: $f_A=f|A$, $g_A=g|A$, $a_A=a_0(f|A)$, $b_A=a_0(g|A)$, $C_\Gamma=[\og_1:\og]$,
$C_k=-(2i)^{k-1}$. Sums over $A$ are over systems of representatives
$A\in\og\backslash
\og_1$. For all $T>1$ we have 
\[
C_k C_\Gamma (f,g)= \sum_A \int_{\F} [f_A(z) \ov{g_A}(z)-a_A\ov{b_A}](z-\ov{z})^w dz
d\ov{z}+a_A \ov{b_A}\Big[\int_{\F_T} (z-\ov{z})^w dz d\ov{z}-C_k\frac{T^{k-1}}{k-1} \Big]    
\]
By Stokes' theorem we find $\int_{\F_T} (z-\ov{z})^w dz d\ov{z}=C_k\frac{T^{w+1}}{w+1}+
\frac{1}{w+1}\int_{\rho^2}^\rho (z-\ov z)^{w+1} d\ov{z}$. In the first integral we apply
Stokes'  theorem after writing $f_A  \ov{g_A}-a_A\ov{b_A}= (f_A- a_A)\ov{g_A}+a_A
(\ov{g_A}-\ov{b_A})$ to get
\[
C_k C_\Gamma (f,g)=\sum_{A} \int_{\partial \F} -F_A(z) \ov{g_A}(z) +
a_A[\ov{g_A}(z)-\ov{b_A}] \frac{(z-\ov{z})^{w+1}}{w+1} d\ov{z}+
\frac{a_A\ov{b_A}}{w+1}\int_{\rho^2}^\rho (z-\ov{z})^{w+1} d\ov{z}
\]
where $F_A(z)=\int_{z}^{i\infty} [f_A(t)-a_A] (t-\ov{z})^w dt$. Since $F_{AT}(z)=F_A(Tz)$, the
integrals over the vertical sides of $\F$ cancel (after summing over $A$) and setting
$\tF_A(z)=F_A(z)-a_A\int_0^z(t-\ov{z})^w dt$ we obtain:
\[
C_k C_\Gamma (f,g)=\sum_{A} \int_\rho^{\rho^2}  \tF_A(z) \ov{g_A}(z)d\ov{z}  
+(-1)^w\frac{a_A}{w+1} \int_{\rho}^{\rho^2} \ov{g_A}(z) \ov{z}^{w+1} d\ov{z}.
\]
In the first integral we change variables $z\rightarrow Sz$, which reverses the order of
integration. As in the proof of Proposition \ref{pa1} we have
$\tF_A(z)-\tF_{AS^{-1}}|_{-w} S (z)=\rho_f(A)(\ov{z})$ obtaining
\begin{equation}\label{7.11}
C_k C_\Gamma (f,g)= \sum_{A}\frac{1}{2} \int_\rho^{\rho^2}  \rho_{f}(A)(\ov{z})
\ov{g_A}(z)d\ov{z}
+(-1)^w\frac{a_A}{w+1}\int_{\rho}^{\rho^2} \ov{g_A}(z) \ov{z}^{w+1} d\ov{z}.
\end{equation}
We now proceed to write the result in terms of the pairing $\lla \cdot ,\cdot \rra$ on
$V_w^\Gamma$. Define $H_{z_0}\in V_w^\Gamma$ as in \eqref{a10}, with $g$ in
place of $f$. Using $\int_{\rho}^{\rho^2} \ov{g_A}(z) (\ov{z}-X)^w
d\ov{z}= \ov{H}_{\rho}(A)-\ov{H}_{\rho^2}(A)$ and \eqref{3.3} we get
$$
\int_{\rho}^{\rho^2} \rho_{f}(A)(\ov{z})
\ov{g_A}(z)d\ov{z}=\Big\langle\rho_{f}(A), \int_{\rho}^{\rho^2} \ov{g_A}(z) (\ov{z}-X)^w
d\ov{z}\Big\rangle=\la \rho_{f}(A), \ov{H}_{\rho}(A)-\ov{H}_{\rho^2}(A) \ra.
$$
The second integral in \eqref{7.11} can be written
\[
\int_{\rho}^{\rho^2} \ov{g_A}(z) \ov{z}^{w+1} d\ov{z}= 
\ov{b_A}\int_{\rho}^{\rho^2}\ov{z}^{w+1} d\ov{z}+\int_{\rho}^{i \infty} -\int_{\rho^2}^{i
\infty}  (\ov{g_A}(z)-\ov{b_A})\ov{z}^{w+1}d\ov{z}
\]
and changing variables $z=t-1$ in the last integral, recalling that
$\rho_f^0(A)=(-1)^w a_A\frac{X^{w+1}}{w+1}$ (=$\rho_f^0(AT)$), and using
\eqref{3.3} we obtain
\begin{gather*}
\sum_A (-1)^w\frac{a_A}{w+1}\int_{\rho}^{\rho^2} \ov{g_A}(z) \ov{z}^{w+1} d\ov{z}= \sum_A
\la \rho_f^0(A)|1-T^{-1} , \ov{H}_\rho(A) \ra +\\ 
\sum_A a_A\ov{b_A} \Big[\int_{0}^1 \int_{0}^\rho
(t-\ov{z})^w d\ov{z} dt +(-1)^w\int_{\rho}^{\rho^2}\frac{\ov{z}^{w+1}}{w+1} d\ov{z} \Big].
\end{gather*}
The expression inside square brackets equals $\frac{1}{(k-1)k}$, and setting $I(f,g)=
\frac{2}{k(k-1)C_\Gamma}\sum_A a_A\ov{b_A}$ we get
\[
2 C_k(f,g)=\lla \rho_f, \ov{H}_\rho-\ov{H}_{\rho^2} \rra + 2\lla \rho_f^0|1-T^{-1},
\ov{H}_\rho\rra+I(f,g).
\]
Now we use 
\begin{equation}\label{8.12}
H_{\rho^2}=H_{\rho}|T-\rho_g^0|(1-T), \quad 
\rho_g=H_\rho-H_{\rho^2}|S=H_\rho|(1-TS) +\rho_g^0|(1-T)S .\end{equation} 
Taking into account the relation
$(1+U^2)=\frac{1}{3}(U^2-U)(1-U^{-1})+\frac{2}{3}(1+U+U^2)$ in $\Q[\og_1]$ we have 
\begin{equation*}
\begin{aligned}
\lla \rho_f, \ov{H}_\rho|1-T \rra&=\lla \rho_f|1-T^{-1}, \ov{H}_\rho \rra = \lla \rho_f|1+S
T^{-1}, \ov{H}_\rho \rra =  \\
&=\frac{1}{3}\lla \rho_f|U^2-U, \ov{H}_\rho|1-U \rra+\frac{2}{3} \lla \rho_f|1+U+U^2, 
\ov{H}_\rho\rra\\
&=\frac{1}{3}\lla \rho_f|\,-T^{-1}\,-TS, \ov{\rho_g}-\ov{\rho_g^0}|(1-T)S \rra+\frac{2}{3} \lla
\rho_f|1+U+U^2, \ov{H}_\rho\rra \\
&=\frac{1}{3}\lla \rho_f |T-T^{-1}, \ov{\rho_g}  \rra +\frac{1}{3}\lla \rho_f|T+T^{-1}S,
\ov{\rho_g^0}|1-T \rra-\\ 
&\ \quad\ \quad\ \quad-\frac{2}{3} \lla
\rho_f^0|(1-T^{-1})(1+U+U^2), \ov{H}_\rho\rra
\end{aligned}
\end{equation*}
(on the last line we used Proposition \ref{pa1}) and collecting terms we get
\begin{equation*}
\begin{split}
6 C_k(f,g)=\lla \rho_f |T-T^{-1}, \ov{\rho_g}  \rra+\lla\rho_f|3+T^{-1}S +T, \ov
\rho_g^0|1-T\rra+\\
+2\lla \rho_f^0|(1-T^{-1})(2-U-U^2), \ov{H}_\rho \rra+3I(f,g)
\end{split}
\end{equation*}
In the second term we use the relation
$$(1-T)(3+ST+T^{-1})=2(T^{-1}-T)+(1-T)S[1+U+U^2-U^2(1+S)]S,$$
while in the third we use $2-U-U^2=(1-U)(1-U^2)$ and  \eqref{8.12}:
\begin{equation}\label{8.13}
\begin{split}
6 C_k(f,g)=\lla \rho_f|T-T^{-1}, \ov{\rho_g} \rra +\lla  \rho_f,
2\ov{\rho_g^0}|(T^{-1}-T)  \rra+\lla 2\rho_f^0|(T-T^{-1}), \ov{\rho_g}\rra+\\
+\lla \rho_f^0|(1-T)(T^{-1}S-ST-3), \ov{\rho_g^0}|1-T \rra+3I(f,g)
\end{split}
\end{equation}
Let $p(X)=\frac{X^{w+1}}{w+1}|1-T$, $q(X)=\frac{X^{w+1}}{w+1}|1-T^{-1}$. Then 
 \[
\lla \rho_f^0|(1-T)(T^{-1}S-ST), \ov{\rho_g^0}|1-T \rra=\frac{1}{C_\Gamma}\sum_A (a_A
\ov{b_{AS^{-1}}}-(-1)^w a_{AS^{-1}}\ov{b_A}) \la p, q|S\ra=0,
\]
where we changed $A$ to $AS^{-1}$ in one of the sums and used that $a_{AJ}=(-1)^w a_A$. We
also have
\[\lla \rho_f^0|(1-T),\ov{\rho_g^0}|1-T \rra=\frac{1}{C_\Gamma}\sum_{A} a_A\ov{b_A}\la
p,p \ra=\frac{1+(-1)^w}{k(k-1)C_\Gamma}\sum_{A} a_A\ov{b_A},
\]
which vanishes if $w$ is odd, and equals $I(f,g)$ if $w$ is even. Therefore the second line in
\eqref{8.13} vanishes if $k$ is even, and it equals $3 I(f,g)=I_k(\rho_f^0, \ov{\rho}_g^0)$ if
$k$ is
odd, finishing the proof.

b) Going backwards in the proof of part a) up to the first equation after applying Stokes'
theorem, we obtain
\[
 C_\Gamma\{\wr_f,\wr_g\}=-6 \sum_{A} \int_{\partial \F}  \tf(A)(z) g_A(z)d z 
\]
with $\tf$ defined in \eqref{a1}. Since the integrand is holomorphic and vanishes at
$i\infty$, each term vanishes. 

c) Since the extended pairing $\{\cdot, \cdot\}$ behaves as the original one under the action
of $\epsilon$, the claim follows
from a) and b) as in the proof of Theorem \ref{thm_main}.
\end{proof}

\subsection{Hecke operators.} \label{sec7.1}
For a finite index subgroup $\Gamma$ and a double coset $\Sigma_n$ satisfying \eqref{eq_star},
we define the operation $|_{\Sigma}$ of the Hecke operators $\wT_n$ on 
$\wW_w^\Gamma$ as in Section \ref{sec4}. Although matrices in the
definition of $\wT_n$ do not preserve $\wV_w^\Gamma$, we have the
following generalization of Proposition \ref{p4.2} and of Corollary \ref{c4.3}.
\begin{prop} \label{pa5} Assume the pair $(\Gamma, \Sigma_n)$ satisfy \eqref{eq_star},
and for part (b) assume that $\Sigma_n=\Sigma_n'$ and $\Gamma$ is normalized by $\epsilon$. Let $\wT_n\in R_n$ be
any element satisying \eqref{hecke}. 

$\mathrm{(a)}$ We have $\wr_{f|[\Sigma_n]}=\wr_f |_{\Sigma} \wT_n$ for $f\in M_k(\Gamma)$.  

$\mathrm{(b)}$  We have ${\wr_{f|[\Sigma_n]}}^\pm={\wr_f}^\pm |_{\Sigma} \wT_n$ for $f\in M_k(\Gamma)$.

$\mathrm{(c)}$ The operators $\wT_n$ preserve the space $\wW_w^\Gamma$. 
\end{prop}  
\begin{proof} (a) The proof is the same as of Prop. \ref{p4.2}, once we show that
$\tf|_{\Sigma}T_n^{\infty}=\widetilde{f| [\Sigma_n]}$. Equation \eqref{4.3} becomes
\begin{equation*}\label{a5}
\begin{split}
\tf|_{\Sigma}T_n^{\infty}(A)&=\sum_{M\in M_n^\infty\cap \Gamma_1 \Sigma_n A }\int_{M z}^{i \infty} 
[f|A_M (t)-a_0(f|A_M)](t-Mz)^w j(M,z)^w dt\\
&=n^{w+1} \sum_{M\in M_n^\infty \cap \Gamma_1 \Sigma_n A}\int_{z}^{i \infty}\big[f|M_{A} A
(u)-a_0(f|A_M)j(M,u)^{-k}\big](u-z)^w du
\end{split}
\end{equation*} 
As in Proposition \ref{p4.2}, we obtain $\tf|_{\Sigma}T_n^{\infty}(A)=\int_{z}^{i
\infty}\big[(f|[\Sigma_n])|A-c(n,f,A) \big](u-z)^w du$ where $c(n,f,A)$ is the sum of the terms
involving $a_0(f|A_M)$ (which is independent of $u$ since $j(M,u)=d_M$ for $M\in M_n^\infty$). Since
the integral converges, we must have $c(n,f,A)=a_0(f|[\Sigma_n]|A)$ (which can alsp be proved directly, using
Prop. \ref{pa1}), hence the last expression equals $\widetilde{f|[\Sigma_n]} (A)$.

(b) The proof is the same as of Corollary \ref{c4.3}.

(c) This follows from part (a) and the decomposition \eqref{7.6}
\end{proof}

\begin{prop}\label{pa6} Assume the hypotheses of Prop. \ref{pa5} and furthermore that $\Sigma_n^\vee$ 
satisfies \eqref{eq_star}. We have for all $\wP,\wQ\in\wW_w^\Gamma$
\[
\{\wP|_{\Sigma}\wT_n, \wQ\} = \{\wP, \wQ|_{\Sigma^\vee}\wT_n \}.
\]
\end{prop}
\begin{proof} 
As in the first proof of Theorem \ref{thm_equiv}, we decompose
$\wP=R+\rho_f^+ +\rho_g^-+\wr_e$ with $R\in C_w^\Gamma$, $f,g\in S_k(\Gamma)$, $e\in M_k(\Gamma)$. Taking into
account Theorem \ref{pa4}, Proposition \ref{pa5} and the fact that the adjoint of the operator $[\Sigma_n]$ is
$[\Sigma_n^\vee]$ with respect to the extended Petersson inner product on $M_k(\Gamma)$ \cite{PP12}, it remains to
show that 
\begin{equation}\label{8.14}
 \{R|_{\Sigma}\wT_n, \wr_e\} = \{R, \wr_{e|[\Sigma_n^\vee]} \}.
\end{equation}

We use Prop. \ref{l7.3}. Let $R=R'|(1-S)$, with $R'(A)=c(A)\in \C$ and $R'|(1-T)=0$. By \eqref{hecke} 
$R|_{\Sigma}\wT_n= R'|_{\Sigma} T_n^\infty (1-S)$, and for $A\in \Gamma_1$ we have
$R'|_{\Sigma} T_n^\infty(A)= \sum_{M\in M_n^\infty} c(A_M)d_M^w $  
where $d_M$ is the lower right entry of $M$, and $MA^{-1}=A_M^{-1} M_A$ with $A_M\in \Gamma_1$, $M_A\in \Sigma_n$.
The left side of \eqref{8.14} becomes (up to a constant which we ignore in the right side as well)
\[
\text{LHS}= \sum_{M\in M_n^\infty} \sum_{\substack{A\in \Gamma\backslash \Gamma_1\\
MA^{-1}\in A_M^{-1}\Sigma_n }}c(A_M) a_0(e|A)d_M^w.
\]
Since $c(AT)=c(A)$, $a_0(e|AT)=a_0(e|A)$, we can replace $M$ by $T^i M T^j$ in the interior sum without changing
it. Therefore, if we write $M=M_{a,d,b}=\left(\begin{smallmatrix}
a & b \\ 0 & d \end{smallmatrix}\right)$ and fix $d$, the interior sum depends only on $b$ modulo
$g_{a,d}=\gcd(a,d)$ and we have
\[\text{LHS}= \sum_{\substack{d|n \\ b \text{ mod } {g_{a,d}}}} \sum_{\substack{A\in \Gamma\backslash \Gamma_1\\
M_{a,d,b}A^{-1}\in A_{M}^{-1}\Sigma_n }}c(A_M) a_0(e|A)\frac{d^{w+1}}{g_{a,d}}.
\]

For the right side of \eqref{8.14}, from the proof of Prop. \ref{pa5} we have for $B\in \Gamma_1$
$$a_0(e|[\Sigma_n^\vee]|B)=\sum_{\substack{M\in M_n^\infty\\ MB^{-1}\in B_M^{-1} \Sigma_n^\vee }}
a_0(e|B_M)\frac{n^{w+1}}{d_M^k} .$$
Since  $MB^{-1}\in B_M^{-1} \Sigma_n^\vee\iff M^\vee B_M^{-1} \in B^{-1} \Sigma_n$, the right side of \eqref{8.14}
becomes, after interchanging $B_M$ and $B$ 
\begin{equation*}
\begin{split}
\text{RHS} &=\sum_{M\in M_n^\infty}\sum_{\substack{B\in \Gamma\backslash \Gamma_1\\
M^\vee B^{-1}\in B_M^{-1}\Sigma_n }}c(B_M) a_0(e|B) \frac{n^{w+1}}{d_M^k} \\
&=\sum_{\substack{d|n \\ b \text{ mod } {g_{a,d}}}}\sum_{\substack{B\in \Gamma\backslash \Gamma_1\\
M_{a,d,b}^\vee B^{-1}\in B_{M}^{-1}\Sigma_n }} c(B_M) a_0(e|B) \frac{a^{w+1}}{d} \frac{d}{g_{a,d}}
\end{split}
\end{equation*}
where the second line follows by writing $M=M_{a,d,b}$ as before.  Comparing the expressions obtained for RHS and
LHS finishes the proof of \eqref{8.14}. 
\end{proof}
From the duality in Proposition \ref{l7.3} and from Propositions \ref{pa5}, \ref{pa6}  we immediately obtain:
\begin{corollary}\label{c8.9}
Assume that both $\Sigma_n$ and 
$\Sigma_n^\vee$ satisfy property \eqref{eq_star}. There exist bases of $\mathcal{E}_k(\Gamma)$ and $C_w^\Gamma$
such that the matrix of the operator $[\Sigma_n^\vee]$ acting on $\mathcal{E}_k(\Gamma)$ is the same as the matrix
of the operator $|_\Sigma \tilde{T}_n$ acting on $C_w^\Gamma$.
\end{corollary}

As an application, we let $\Gamma=\Gamma_1(N)$, and $\chi$ a character modulo $N$, and we show 
that the trace of Hecke operators $T_n$ on the Eisenstein subspace
$\EE_k(N, \chi)\subset M_k(N,\chi)$ is the same as the trace of $\wT_n$ on $C_w^{\Gamma,\chi}$ (see
$\S$\ref{sec5.2} for the notation). For $\Gamma_1$, when $C_w^{\Gamma_1}=<X^w-1>$, a direct proof is immediate, but
for $\Gamma=\Gamma_1(N)$ it seems difficult to prove the statement without using the dual space $\wE_w^\Gamma$ and
the pairing $\{\cdot, \cdot\}$.
\begin{prop}\label{p8.9} $\mathrm{(a)}$ Let $\Gamma=\Gamma_1(N)$ and 
let $\wE_w^{\Gamma,\chi}\subset \wW_w^\Gamma$ be the image of the Eisenstein subspace
$\EE_k(N,\chi)\subset M_k(\Gamma)$ under the map $f\rightarrow \wr_f$. For $(n,N)=1$ we have 
\[ 
\tr(\EE_k(N, \chi)|T_n)=\tr(\wE_w^{\Gamma,\chi}|_{\Delta} \wT_n)=\tr(C_w^{\Gamma,\chi}|_{\Delta} \wT_n).
\] 

$\mathrm{(b)}$ For $\Gamma=\Gamma_0(N)$ and $n\|N$, let $\Theta_n$  be the double coset and let $W_n$ be
the Atkin-Lehner operator defined in Section \ref{sec5.11}. We have
\[ 
\tr(\EE_k(\Gamma)|W_n)=\tr(\wE_w^{\Gamma}|_{\Theta} \wT_n)=\tr(C_w^{\Gamma}|_{\Theta} \wT_n).
\]
\end{prop}
\begin{proof} 
(a) The duality in Prop. \ref{l7.3} between $C_w^\Gamma$ and $\wE_w^{\Gamma}$
with respect to the pairing $\{\cdot, \cdot\}$ implies dualities between $C_w^{\Gamma,\chi}$ and $\wE_w^{\Gamma,
\ov{\chi}}$. Therefore, taking $\Sigma_n=\Delta_n$ for $(n,N)=1$ in Corollary \ref{c8.9}, it follows that the
eigenvalues of $|_\Delta \wT_n$ on $C_w^{\Gamma,\chi}$ are the same as the eigenvalues of $T_n^*=[\Delta_n^\vee]$
on $\EE_k(N, \ov{\chi})$, which are the same as the eigenvalues of $T_n$ on $\EE_k(N, \chi)$ (the latter
space has a basis of eigenforms for $T_n$ with $(n,N)=1$). 

(b) The claim follows from Corollary \ref{c8.9}, using the fact that $\Theta_n=\Theta_n^\vee$.
\end{proof}
\begin{remark} \label{r8.10}
Prop \ref{p8.9} shows that for $\Gamma=\Gamma_1(N)$ and $(n,N)=1 $ we have
$$\tr(W_w^{\Gamma,\chi}|_{\Delta}\wT_n)=\tr(M_k(N,\chi)|T_n)+\tr(S_k(N, \chi)|T_n),$$
and the same for Atkin-Lehner operators on $\Gamma_0(N)$.  
For $\Gamma=\Gamma_1$, this fact was an ingredient used by Zagier to sketch an elementary proof of
the Eichler-Selberg trace formula, by computing directly the left side for an
appropriately chosen $\wT_n$ \cite{Za93}. A generalization of this approach giving a simple trace
formula for $M_k(N,\chi)$ is work in progress of the second author and Don Zagier. 
\end{remark}

\subsection{Extra relations revisited}\label{sec7.2}

Theorem \ref{pa4} gives another way of determining the extra relations satisfied by all
period polynomials of cusp forms which are independent of the period relations.   
Assuming that $\wr^-$ is an isomorphism (see Proposition \ref{p7.4}), it follows that there
exist $g\in \EE_k(\Gamma)$ such that  $\wr^-_g$ form a basis for $(\wE_w^\Gamma)^-$. Since the
pairing $\{\cdot ,\cdot \}$ is nondegenerate, it follows that the linear relations
$\{P, \overline{\wr^-_g} \}=0$ are satisfied by $P=\rho_f^{+}$, for all $f\in S_k(\Gamma)$,
but they are not satisfied by some $P\in (C_w^\Gamma)^+$. A similar argument applies to
determine the relations satisfied by $\rho_f^{-}$, when $(C_w^\Gamma)^-\ne 0$ and $\wr^+$ is an
isomorphism. These linear relations can be used to define other versions of the linear forms
$\lambda_+, \lambda_-$ in Proposition
\ref{p6.2}, which are entirely explicit once the period polynomials of Eisenstein series are
determined.

As an example we take $\Gamma=\Gamma_1(N)$, and we assume $k\ge 3$. Then the Eisenstein
subspace $\EE_k(\Gamma)$ has a basis of Eisenstein series which are Hecke eigenforms for the
Hecke operators of index coprime with the level \cite[Ch. 5]{DS}. Their period polynomials for
the identity coset can be determined in terms of special values of Dirichlet $L$-functions  by
Proposition \ref{pa1}. For other cosets $A\in \Gamma\backslash \Gamma_1$ the period
polynomials of the Hecke eigenforms are harder to compute. Instead, consider a second
basis, consisting of Eisenstein series which vanish at all but one cusp, so that the action
$|A$ permutes the elements of this basis. The elements of the second basis can be decomposed in
terms of Hecke eigenforms, so their period polynomials corresponding to all cosets $A$ can be
determined explicitly. For $\Gamma_0(N)$ we are planning
to return to this question in a future work.

\noindent\textbf{\small Acknowledgments.} Part of this work was completed at the
Max Planck Institute in Bonn, which provided financial support and a great working
environment. We would like to thank Don Zagier for inspiring conversations. The
first author was partially supported by the CNCSIS grant PN-II-RU-TE-2012-3-0455. The second
author was partially supported by the European Community grant PIRG05-GA-2009-248569.

\end{document}